\documentclass[12pt,fleqn]{article}
\usepackage{amsmath,amssymb}
\usepackage{multicol}
\usepackage{amsthm}
\usepackage{graphicx}
\usepackage{epstopdf}
\usepackage{color}
\usepackage{hyperref}
\usepackage{natbib}
\usepackage{float}
\usepackage[utf8]{inputenc}
\usepackage{enumerate}
\usepackage[multiple]{footmisc}
\usepackage[table]{xcolor}
\usepackage{subcaption}
\usepackage{booktabs}
\usepackage{array}
\usepackage{ulem}

\usepackage[english]{babel}
\usepackage[autostyle, english=american]{csquotes} % smart quotes
\MakeOuterQuote{"} % Use " for outer quotes

\allowdisplaybreaks

\numberwithin{equation}{section}

\makeatletter
\renewcommand{\section}{\@startsection{section}{1}{0pt}{20pt}{6pt}{\large\bf}}
\renewcommand{\@seccntformat}[1]{\csname the#1\endcsname.\ }
\makeatother

\def\footnoterule{\kern -3pt \hrule width 2.7 true cm \kern 2.6pt}

\newcommand{\nada}[1]{}

\newtheorem{theorem}{Theorem}[section]
\newtheorem{lemma}[theorem]{Lemma}

\theoremstyle{remark}

\newtheorem{remark}[theorem]{Remark}

\newtheorem{example}[theorem]{Example}

\begin{document}

\title{\textbf{On the First Hitting Time Problems \\
for Diffusion Processes: \\
Local Time-Space Approach}}

\author{\large{Jerome Detemple}\thanks{Questrom School of Business, Boston University, Boston, MA, USA}, \large{ Yerkin Kitapbayev}\thanks{Mathematics Department,
Khalifa University of Science and Technology, PO Box 127788, Abu Dhabi, United Arab Emirates }, 
\large{ Danila Shabalin}\thanks{Faculty of Mechanics and Mathematics, Lomonosov Moscow State University, Moscow, Russia} \thanks{Vega Institute Foundation, Moscow, Russia}
}

%\date{}

\maketitle

%%%%%%%%%%%%%%%%%%%%%%%%%%%%%%%%%%%%%%%%%%%%%%%%%%%%%%%%%%%%%%%%%%%%%%%%%%%%%%%
%%% Abstract %%%
%%%%%%%%%%%%%%%%%%%%%%%%%%%%%%%%%%%%%%%%%%%%%%%%%%%%%%%%%%%%%%%%%%%%%%%%%%%%%%%

%{\par \leftskip=2.6cm \rightskip=2.6cm \footnotesize

\begin{abstract}
Using the local time-space calculus of \cite{P-2005} and the method developed in \cite{M-2010}, we derive a new integral representation for the backward probability associated with the first-passage time (FPT) of a diffusion process through a time-dependent barrier. We present a complete three-step numerical algorithm: first, the problem is reduced to a Volterra-type integral equation; second, its kernel is approximated by a Markov chain; finally, the resulting equation is solved using a quadrature rule. The method is implemented for several representative examples, and its convergence properties are established. An extension to double barrier problems is carried out.
\end{abstract}

\section{Introduction}

The determination of the density of the first-passage time (FPT) to a barrier is a problem of longstanding interest. From a purely theoretical point of view, the problem gives rise to non-trivial mathematical characterizations, raising questions of existence, uniqueness, and numerical computation of solutions. From an application perspective, the density of the FPT plays a key role in various areas, including Biology, Chemistry, Physics, Finance, and Economics. In this paper, we provide a new characterization of the distribution of the FPT of a diffusion process to a time-dependent barrier and propose a numerical algorithm for its resolution. 

Our contributions can be summarized as follows. First, we characterize the backward probability associated with the FPT of a diffusion process to a continuously differentiable lower barrier as a functional involving the solution to a linear Volterra integral equation (VIE) of the first kind. The underlying diffusion is arbitrary, subject only to standard conditions ensuring the existence of a unique strong solution and a smooth transition density function (TDF). The resulting backward hitting probability is then represented as the difference of two terms. The first term is the probability that the terminal value of the underlying diffusion lies below the terminal barrier level. The second term accounts for barrier crossings during the remaining time interval and is expressed as an integral with respect to the local time of the process at the barrier. Its integrand involves the right derivative of the backward hitting probability. The standard cumulative distribution function of the hitting time can be recovered by evaluating the backward hitting probability for different values of the terminal time.

Second, we extend the analysis to the double-barrier setting with upper and lower barriers. As in the single-barrier case, the backward hitting probability admits a decomposition into two terms. The first is the probability that the terminal value of the process lies outside the barrier region. The second is a local time component involving the spatial derivatives of the backward hitting probability at the upper and lower barriers. These derivatives satisfy a system of coupled first-kind VIEs.

 We also note that the results above are related to the work of  \cite{M-2010}, where the local time-space calculus of \cite{P-2005} is instrumental in deriving pricing formulas for double-barrier options under one-dimensional time-homogeneous diffusions. Since barrier-option prices are naturally expressed in terms of first-hitting time distributions, some of the integral identities derived in the present paper can, after suitable normalization and reformulation, be recovered from the equations in \cite{M-2010} when the diffusion coefficients are time-independent. Our paper extends this framework to the time-inhomogeneous case and advances the theory in several directions, as explained below. 
 
Third, we propose a numerical scheme for solving the VIE and establish its convergence and error properties. The scheme relies on a regularization of the integral equation through a multiplicative square-root-of-time transformation. The resulting VIE is solved using the product integration method of \cite{L-1985}. We demonstrate the approach by computing the hitting time distributions for the Bessel, CEV, and Feller processes. For diffusions whose transition densities are unavailable in closed form or are computationally intractable, we employ a Markov chain approximation of the underlying diffusion.

Fourth, we derive for the geometric Brownian motion (GBM) with constant parameters a Volterra integral equation of the second kind for the spatial derivative, evaluated at the barrier, of the hitting distribution function. The same argument can be extended with only minor modifications to linear stochastic differential equations. This second-kind formulation is particularly attractive computationally since VIEs of the second kind are typically better conditioned and more numerically stable than their first-kind counterparts. Extending this formulation to general diffusions is left for future work. 

Fifth, we extend the method to a class of stochastic volatility models with a multiplicatively separable volatility coefficient. The state variable in volatility follows a correlated diffusion process. In this instance, we apply a conditional Monte Carlo method, inspired by \cite{W-1997} and \cite{LP-2009}. Conditional on the trajectories of the state variable to the terminal date, the backward probability associated with  the first-hitting time of a barrier satisfies the general formula with the associated VIE described above. The backward probability of the FPT  can then be calculated by taking the expectation with respect to the Wiener measure associated with the Brownian motion (BM) driving the evolution of the state variable. For computation, this expectation is estimated by averaging over trajectories of the state variable over the entire period.

The paper relates to several branches of the literature. A large body of literature characterizes the distribution of the first-passage time using a Volterra integral equation, as in this paper. The pioneering contribution in this direction is due to \cite{F-1943}, who derives a VIE of the first kind for the FPT density of a Markovian process. \cite{S-1951} establishes a connection between VIEs in the time domain and differential equations in the Laplace domain. \cite{D-1971} shows how to transform the VIE of the first kind into a second-kind VIE for Brownian motion, facilitating numerical computation. \cite{PS-1976} extend the analysis to curved boundaries and provide an efficient discretization scheme for computation. \cite{BNR-1987} develop an algorithm, the "BNR algorithm", for the case of Brownian motion, generalizing the regularization of the VIE. \cite{RS-1988} provide a detailed study of the FPT density and its moments for the Ornstein-Uhlenbeck (OU) process with a constant boundary, including asymptotic analysis. \cite{GNRS-1989} adapt the BNR algorithm to handle one-dimensional diffusions. Continuing Durbin's earlier contribution to curved boundaries, \cite{DW-1992} expand the FPT density as a series of multiple integrals, providing error bounds for truncation, which allows for efficient computation. \cite{GRRT-1997} extend the VIE approach of \cite{GNRS-1989} to time-inhomogeneous diffusions. Furthermore, \cite{ST-1996} use series expansions based on a fixed-point theorem to solve this VIE and provide error bounds for the resulting approximations.

Further advances in this century exploiting VIEs pertain to theory and numerical implementation. \cite{P-2002} derives a system of integral equations from a single "master equation" that can be viewed as a Chapman-Kolmogorov equation of Volterra type. \cite{DNPR-2001} simplify the VIE for the case of Gauss-Markov processes, facilitating computation. \cite{RST-2008} develop a system of integral equations for the survival probability and the conditional distribution of the process at the crossing time, providing additional information about the process behavior. Focusing on Cherkasov (reducible) processes, \cite{LK-2020} use the heat potential approach to derive a new VIE for the FPT density. \cite{A-2003} allows for stochastic boundaries and simplifies the VIE by approximating the underlying diffusion. \cite{ZS-2009} consider the inverse problem, namely finding the boundary given the FPT distribution. To that end, they develop a numerical algorithm based on the discretization of the VIE. \cite{BSZ-2013} extend the VIE approach to a bivariate degenerate diffusion process, specifically integrated Brownian motion and the integrated OU process.

Although the literature on VIE-based approaches is voluminous, there are other methods for characterizing and computing the density of the FPT. Alternatives that have been explored include weak approximations, such as \cite{G-2000} and \cite{LB-2023}; Monte Carlo simulation, e.g., \cite{IK-2011} and \cite{HZ-2019}; PDE methods, e.g., \cite{PW-2008} and \cite{BCGS-2021}; series expansions applied to mean-reverting processes such as the OU and CIR models, e.g., \cite{L-2004} and \cite{APP-2005}; and Fourier transform methods applied to the Heston model, e.g., \cite{MP-2009}. One should also mention the substantial literature that applies these methods in other contexts. Applications in the Natural Sciences include, for instance, \cite{R-2001}, \cite{HCB-2010}, and \cite{ZD-2016}. Applications in Finance comprise the early work of \cite{BC-1976}, and subsequent contributions to corporate finance, e.g., \cite{L-1994}, credit risk, e.g., \cite{AZ-2001}, \cite{LK2-2020}, and barrier option pricing, e.g., \cite{M-2010}, \cite{IM-2022}, \cite{DKS-2025}.  

Finally, since our numerical algorithm relies, in some cases, on a Markov chain approximation (MCA) of the transition density of the underlying diffusion, one may also mention alternative approximation approaches, such as Hermite polynomial expansion (\cite{BEZ-2025}), lower-upper bound approximation (\cite{GDD-2004}), heat kernel expansion (\cite{PP-2012}), and neural networks (\cite{STN-2025}).

The paper is organized as follows. Section~\ref{S:model} describes the model. Section~\ref{S:single} derives the VIE for the backward probability  of the FPT in the case of a single barrier. Section~\ref{S:double_A} derives equations for cases involving double barriers.  Section~\ref{S:properties} shows the existence and uniqueness of the solution and presents a numerical procedure for computation. Section~\ref{sec:VIESK} derives the VIE of the second kind. Section~\ref{S:MCA} describes a Markov chain approximation for numerical computation and provides an application of the VIE-MCA approach to a Feller process. Section~\ref{S:volatility} extends the analysis to a class of stochastic volatility models. Results on the convergence of the numerical scheme are established in the Appendix~\ref{S:APP}.

\section{Preliminaries}
\label{S:model}
Consider a stochastic basis $(\Omega, \mathcal{F}, \mathbb{P}; \mathbb{F} = (\mathcal{F}_t)_{t \ge 0})$ 
where the filtration $\mathbb{F}$ satisfies the usual conditions  (right-continuity and completeness with respect to $\mathbb{P}$--null sets). On this space, let us assume a one-dimensional diffusion process $X = \left(X_t\right)_{t\ge 0}$ with drift $\mu: \mathbb{R}_+\times \mathbb{R} \to \mathbb{R}$ and diffusion coefficient $\sigma: \mathbb{R}_+ \times \mathbb{R} \to \mathbb{R}$, governed by a stochastic differential equation (SDE):
\begin{equation} \label{SDE}
    d X_t= \mu(t,X_t) \,dt + \sigma(t,X_t) \,d W_t, \qquad X_0 = x \in \mathbb{R},
\end{equation}
where $W = \left(W_t\right)_{t\ge 0}$ is a standard $\mathbb{F}$--Brownian motion under the measure $\mathbb{P}$. We suppose that the SDE \eqref{SDE} admits a unique strong solution. To guarantee this, one may impose the It\^o or the Yamada and Watanabe conditions (see \cite{KZ-1981}). The solution is Markovian, with transition probability $\mathbb{P}_{t,x}(X_s \in A) = \mathbb{P}(X_s \in A \mid X_t = x)$, defined for $0 \le t < s$ and any Borel set $A \subset \mathbb{R}$. Furthermore, we assume the existence of the transition density function $p(s,y;t,x)$ determined by the relation $\mathbb{P}(X_s \in A \mid X_t = x) = \int_A p(s, y; t, x)\,dy$. For instance, this holds if the coefficients are H\"older continuous and the diffusion term $\sigma(t,x)$ is uniformly elliptic: there exists an $\epsilon >0$ such that $\sigma^2(t,x) \ge \epsilon$ for all $(t,x)$ in the domain (see  Chapter~5 in \cite{KS-1991}), although this condition is far from necessary. We further impose the smoothness of the transition density, ensured by additional regularity conditions on the coefficients (see Chapter~1 in \cite{F-1964}). 
We define the infinitesimal generator $\mathbb{L}_X$ of $X$ as the following second-order linear differential operator
\begin{equation}
\mathbb{L}_X  = \mu(t,x)\, \partial_x + \frac{1}{2} \sigma^2(t,x) \,\partial^2_{xx}.   
\end{equation}

We also introduce the following standard notation. For $y\in\overline{\mathbb{R}}$, we write $f(x)\sim g(x)$ as $x\to y$ if $\lim_{x\to y} f(x)/g(x)=1$, and $f(x)=\mathcal{O}(g(x))$ as $x\to y$ if there exists a constant $C>0$ and a neighborhood of $y$ such that $|f(x)|\le C\,|g(x)|$ for all $x$ in that neighborhood.

\section{First Hitting Time: Single Barrier}
\label{S:single}

Let us consider a time-dependent lower barrier $b: \mathbb{R}_+ \to \mathbb{R}$. To apply the local time-space formula below and ensure the existence and uniqueness of the resulting equations, we assume that $b \in C^1\left(\mathbb{R}_+\right)$.
We define the first-hitting time
\begin{equation}
    \tau_b= \tau_b(t)=\inf\{s\ge t:X_s\le b(s)\}
\end{equation}
with $X_t>b(t)$ for starting time $t\ge 0$. To avoid trivial cases, only attainable boundaries are considered, i.e., those for which $ \mathbb{P}_{t,x}(\tau_b < \infty) > 0$. We additionally assume that $\sigma(t,b(t))\neq0$ for all $t\ge 0$.
\begin{example}
A zero boundary $b=0$ is not attainable for a geometric Brownian motion starting from $X_0 > 0$. To validate the attainability of a constant barrier $b$ in the homogeneous case, one can apply Feller's test for explosions (\cite{KS-1991}, p.~342).
\end{example}

Let $T>0$ be an arbitrary but fixed terminal time. The primary object of our analysis is the hitting probability $G^T(t, x):=\mathbb{P}_{t,x}(\tau_b\le T )$, $0\le t\le T$, which is the probability of hitting the boundary $b$ before the terminal time $T$, conditional on $X_t=x$. This fixed-horizon formulation (backward problem) is motivated by applications in mathematical finance, such as barrier option pricing and credit risk, as well as in physics, where absorption and escape events are studied over a prescribed observation time. The classical cumulative distribution function (CDF) of the hitting time is then recovered by evaluating $G^T(0,x)$ for different values of the terminal time $T$. In what follows, we omit the superscript $T$ for simplicity.

The following theorem provides a semi-analytical representation in the spirit of \cite{M-2010}, where related representations were derived for barrier-option prices under the assumption of time-homogeneous diffusions. More precisely, the survival probability equals the expected payoff
of a no-touch claim paying one at maturity if the barrier is not
hit. Under a pricing measure and deterministic interest rates,
discounting this probability gives the corresponding option price.
The hitting probability is the complement of the survival
probability. Consequently, for time-homogeneous diffusions within the class
treated in \cite{M-2010}, the corresponding specializations of the
following two theorems can be recovered from its single-barrier
local time representation, after taking complements and reconciling
the kernel normalizations. For completeness, and to adapt the argument to the notation and first-hitting time formulation used here, we provide a concise proof of the first result. 
In Section~\ref{sec:VIESK}, we derive a VIE of the second kind for the function $f(t)$ introduced below, a result that is not contained in \cite{M-2010}.

\begin{theorem} \label{cdf-single}
The backward hitting probability $G(t,x)$ associated with the lower  boundary $b(t)$ is given by
\begin{equation}\label{G-single}
   G(t,x)=\mathbb{P}_{t,x}\left(X_T\le b(T)\right)-\frac{1}{2}\int_{t}^{T} f(u)\, p(u, b(u) ; t, x)\, du
\end{equation}
for $t\in[0,T]$, where $f(t)=G_x(t, b(t)+) \,\sigma^2(t, b(t))$.
\end{theorem}
\begin{proof}
First, we reformulate the original problem for $G(t,x)$ as the following backward Kolmogorov equation; see \cite{K-1931}:
\begin{alignat}{2}
   & G_t(t,x) + \mathbb{L}_X G(t,x)=0, \qquad &&t\in[0,T),\; x>b(t), \\
   & G(t,x) =1, &&t\in[0,T],\; x\le b(t),\\
   & G(T,x) =0,  &&x>b(T).
\end{alignat}
We apply the change-of-variable formula of \cite{P-2005} to the process $G(s,X_s)$, $t\le s\le T$. This yields
\begin{align}
   G(T,X_T)
=&\,G(t,x)+\int_t^T\!\left(G_t(u,X_u)+\mathbb{L}_XG(u,X_u)\right)\,du+M_T \\
   &+\frac{1}{2}\int_t^T G_x(u,b(u)+)\,d\ell_u^{b}(X)\nonumber\\
   =&\,G(t,x)+M_T
   +\frac{1}{2}\int_t^T G_x(u,b(u)+)\,d\ell_u^{b}(X),\nonumber
\end{align}
where $M=\left(M_t\right)_{t\ge0}$ is a zero-mean martingale and $\ell^{b}(X)$ denotes the local time of $X$ at the curve $b(\cdot)$, defined by
\begin{equation}
 \ell_t^b(X)
 =
 \lim_{\varepsilon\to0+}
 \frac{1}{2\varepsilon}
 \int_0^t
 \mathbf{1}_{\left\{\left|X_s-b(s)\right|<\varepsilon\right\}}\,
 d\langle X,X\rangle_s.
\end{equation}
Recalling the terminal condition for $G(T,\cdot)$, taking expectations $\mathbb{E}_{t,x}[\cdot]$ on both sides and using the
martingale property of $M$, we obtain
\begin{equation}
    G(t,x)
    =
    \mathbb{P}_{t,x}(X_T\le b(T))
    -\frac{1}{2}
    \int_t^T
    G_x(u,b(u)+)\,
    d\mathbb{E}_{t,x}\!\left[\ell_u^{b}(X)\right].
\end{equation}
By the occupation time formula for $X$, the expected local time admits the representation
\begin{equation}
\mathbb{E}_{t,x}\!\left[\ell_u^{b}(X)\right]
=
\int_t^u
\sigma^2(s,b(s))
\,p(s,b(s);t,x)\,ds.
\end{equation}
Hence,
\begin{align}
   G(t,x)
   =&\,
   \mathbb{P}_{t,x}(X_T\le b(T))
   \\
   &-\frac{1}{2}
   \int_t^T
   G_x(u,b(u)+)\,
   \sigma^2\!\left(u,b(u)\right)
   p(u,b(u);t,x)\,du
   \nonumber
\end{align}
for $t\in[0,T]$, which yields the desired decomposition of the backward cumulative distribution function $G(t,x)$.
\end{proof}
The representation \eqref{G-single} is not yet sufficient to solve the problem, since it involves the unknown function $f(t)$. To determine $f(t)$, we evaluate \eqref{G-single} at $x=b(t)$, use the boundary condition, and obtain the following first-kind VIE.
\begin{theorem}
The function $f(t)$ satisfies the linear Volterra integral equation of the first kind
\begin{equation}\label{volterra-single}
\mathbb{P}_{t,b(t)}\!\left(X_T>b(T)\right)
=
-\frac{1}{2}
\int_t^T
f(u)\,
p(u,b(u);t,b(t))
\,du
\end{equation}
for $t\in[0,T)$. The behavior at the terminal point $t=T$ is understood through the left-hand limit as $t\to T-$.
\end{theorem}

The preceding results hold for a fixed terminal horizon $T$, with
$t\in[0,T]$ denoting the starting time. Temporarily restoring the
suppressed dependence on $T$, write $G^T(t,x)$ for the backward
hitting probability. In general, computing the time-zero
hitting-time distribution $T\mapsto G^T(0,x)$ using this formulation
requires solving a horizon-dependent Volterra equation for each
terminal time $T$ of interest.

If, however, the diffusion is time-homogeneous and the barrier is
constant, time-translation invariance gives $G^T(t,x)=G^{T-t}(0,x)$, $0\le t \le T$.
The hitting probability then depends only on the remaining time $T-t$. Consequently,
a single Volterra solve determines the time-zero hitting-time
distribution over the entire interval $[0,T]$, rather than only
at its terminal point. We illustrate this reduction below.

We now consider a constant lower barrier $b$ and a time-homogeneous
diffusion $X$ with the coefficients $\mu(x)$ and $\sigma(x)$. For $X_0=x>b$, define
\begin{equation}
    \tau_b:=\inf\{s\ge0:X_s\le b\},
    \qquad
    F(t,x):=\mathbb{P}_{0,x}(\tau_b\le t).
\end{equation}
Here, $t$ denotes elapsed time, in contrast to the starting time
in the preceding backward formulation.  The following result follows from Theorem~\ref{cdf-single} and time-translation invariance.

\begin{theorem}\label{cdf-single-const}
The cumulative distribution function $F(t,x)$ of the first-hitting
time of the constant lower barrier $b$ admits the representation
\begin{equation}\label{F-single-const}
F(t,x)
=
\mathbb{P}_{0,x}(X_t\le b)
-
\frac12
\int_0^t
f(u)\,p(t-u,b;0,x)\,du
\end{equation}
for $t \in [0,T]$ and $x>b$, where
$ f(t):=F_x(t,b+)\,\sigma^2(b)$.
The boundary function $f$ satisfies the linear Volterra integral
equation of the first kind
\begin{equation}\label{volterra-single-const}
\mathbb{P}_{0,b}(X_t>b)
=
-\frac12
\int_0^t
f(u)\,p(t-u,b;0,b)\,du
\end{equation}
for $t \in (0,T]$, with the behavior at the origin understood in the right-hand limit as $t\to0+$.
\end{theorem}

\begin{example}
\textit{Bessel process of integer dimension $d\ge 2$.}
We assume that the evolution of $X$ is given by
\begin{equation}
dX_t
=
\frac{d-1}{2X_t}\,dt
+
dW_t,
\qquad
X_0=x.
\end{equation}
We recall that a Bessel process is a real-valued stochastic process representing the Euclidean norm (distance from the origin) of a multi-dimensional Brownian motion. The transition density of a Bessel process of dimension $d$ is known explicitly:
\begin{equation}
p_d(s,y;t,x)
=
\frac{y}{s-t}
\left(\frac{y}{x}\right)^\nu
\exp\!\left(
-\frac{x^2+y^2}{2(s-t)}
\right)
I_\nu\!\left(
\frac{xy}{s-t}
\right),
\end{equation}
for $0\le t<s$ and $x,y>0$, where $\nu=d/2-1$ and $I_\nu$
denotes the modified Bessel function of the first kind of
order $\nu$.
Figure~\ref{fig:BESSEL} displays the cumulative distribution
function $F(t,x)$ and its density  for the
first-hitting time of the constant upper barrier $b$, computed using the adaptation of Theorem~\ref{cdf-single-const} to the case of the upper barrier.
\end{example}

\begin{figure}[h]
    \centering
    \begin{subfigure}[t]{0.48\textwidth}
        \centering
        \includegraphics[width=\textwidth]{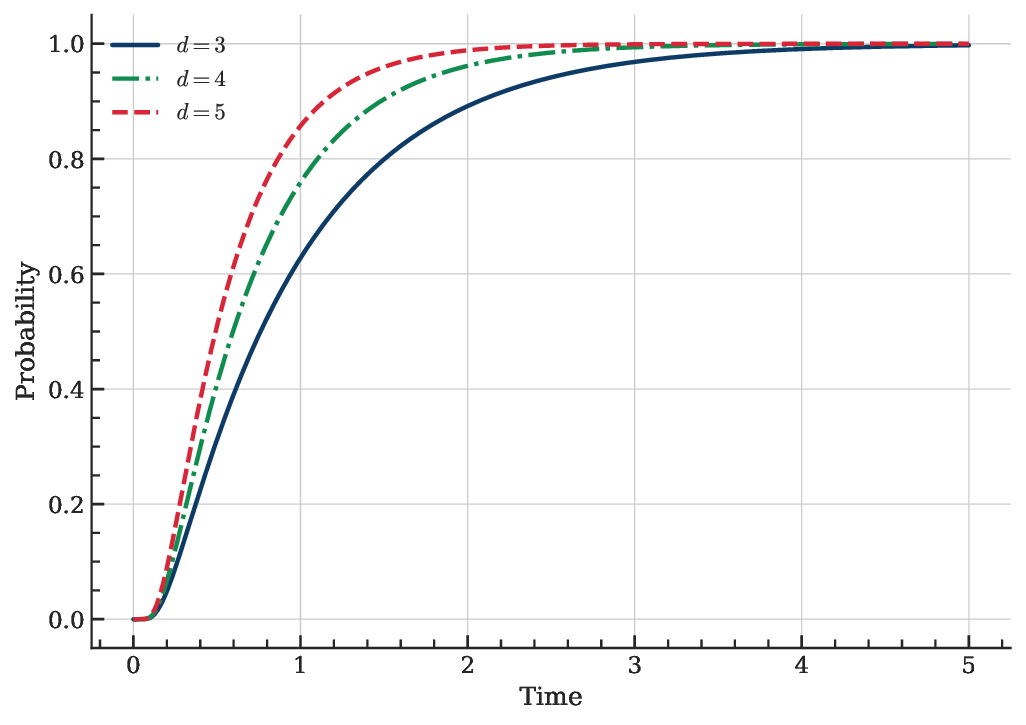}
        \caption{CDF}
        \label{fig:BESSEL-hitting-cdf}
    \end{subfigure}
    \hfill
    \begin{subfigure}[t]{0.48\textwidth}
        \centering
        \includegraphics[width=\textwidth]{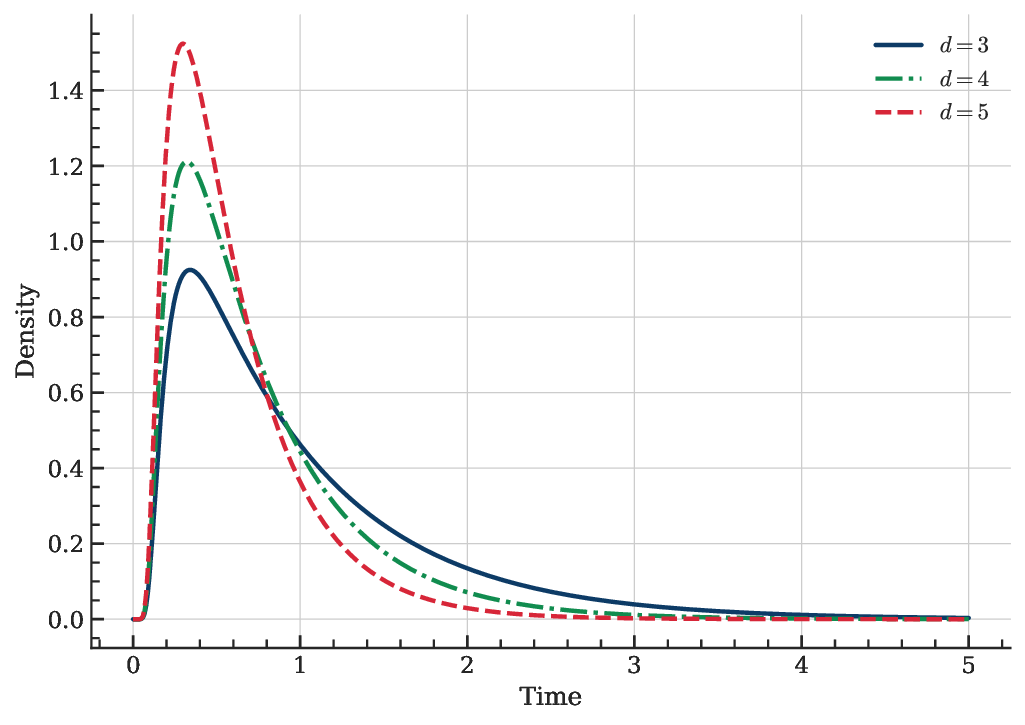}
        \caption{PDF}
        \label{fig:BESSEL-hitting-pdf}
    \end{subfigure}
    \caption{
    Bessel process:
    cumulative distribution function $t\mapsto F(t,X_0)$ (left) and probability density function $t\mapsto F_t(t,X_0)$ (right) of the first-hitting time.
    Effect of dimension $d$.
    Parameters: $T=5$, $X_0=1$, $b=2$ (constant upper barrier). Details of the numerical scheme to solve the corresponding VIE are in Section~\ref{S:properties}. The number of time steps used is $m=2^8$.}
    \label{fig:BESSEL}
\end{figure}

\section{First Hitting Time: Double Barrier}
\label{S:double_A}

In this section, we study the first-hitting time of two barriers. As above, we allow for time-dependent model parameters and barriers. We consider the formulation where it is irrelevant which of the two smooth barriers is hit first. For the process $X$, given two barriers $b_1(t)<b_2(t)$, the first-passage time is defined by 
\begin{equation}
    \tau= \tau(t) = \inf\{s \ge t:X_s\le b_1(s) \text{ or }X_s\ge b_2(s)\},
\end{equation}
with $b_1(t)<X_t=x<b_2(t)$.  We study the following boundary value problem for the backward  hitting probability $G(t, x)=\mathbb{P}_{t,x}(\tau\le T)$ 
\begin{align}
& G_t(t, x) + \mathbb{L}_X G(t,x)=0  ,&& t\in[0,T),\; x\in( b_1(t),b_2(t)) \\
&G(t, x)  =1,&& t\in[0,T],\; x\in(-\infty, b_1(t)]\cup   [b_2(t),+\infty)\\
&G(T, x)  =0, && x\in(b_1(T),b_2(T))
\end{align}
where $T$ is the fixed terminal time.
As before, we apply the local time-space formula of \cite{P-2005} to the process $G(s,X_s)$, $t\le s\le T$,
\begin{align}
   G(T,X_T)
   =&\,G(t,x)
   +\int_t^T\!\left(G_t(u,X_u)+\mathbb{L}_XG(u,X_u)\right)du
   +M_T\\
   &+\frac{1}{2}\int_t^T G_x(u,b_1(u)+)\,d\ell_u^{b_1}(X)
   -\frac{1}{2}\int_t^T G_x(u,b_2(u)-)\,d\ell_u^{b_2}(X)\nonumber\\
   =&\,G(t,x)+M_T
   +\frac{1}{2}\int_t^T G_x(u,b_1(u)+)\,d\ell_u^{b_1}(X)\nonumber\\
   &-\frac{1}{2}\int_t^T G_x(u,b_2(u)-)\,d\ell_u^{b_2}(X),\nonumber
\end{align}
where $M$ is a zero-mean martingale and $\ell_t^{b_i}(X)$, $i=1,2$, denotes the local time of $X$ at the curve $b_i(\cdot)$.
Recalling the terminal condition for $G(T,\cdot)$, taking expectations on both sides, and applying the martingale property, we obtain
\begin{align}\label{F-double-A}
G(t,x)
=&\,\mathbb{P}_{t,x}\!\left(X_T\le b_1(T) \text{ or } X_T\ge b_2(T)\right) \\
&-\frac{1}{2}\int_t^T G_x(u,b_1(u)+)\,d\mathbb{E}_{t,x}\!\left[\ell_u^{b_1}(X)\right] +\frac{1}{2}\int_t^T G_x(u,b_2(u)-)\,d\mathbb{E}_{t,x}\!\left[\ell_u^{b_2}(X)\right]\nonumber\\
=&\,\mathbb{P}_{t,x}\!\left(X_T\le b_1(T)\text{ or } X_T\ge b_2(T)\right)\nonumber\\
&-\frac{1}{2}\int_t^T f_1(u)\,p(u,b_1(u);t,x)\,du + \frac{1}{2}\int_t^T f_2(u)\,p(u,b_2(u);t,x)\,du\nonumber
\end{align}
for $t\in[0,T]$, where $f_1(t):=G_x(t,b_1(t)+)\sigma^2(t,b_1(t))$, $f_2(t):=G_x(t,b_2(t)-)\sigma^2(t,b_2(t))$ are unknown boundary derivatives of $G$, multiplied by the factor $\sigma^2$.
The representation \eqref{F-double-A} does not yet provide a complete solution, since the functions $f_1$ and $f_2$ remain unknown. To determine them, we set $x=b_i(t)$ in \eqref{F-double-A}, invoke the boundary conditions, and obtain the following system of linear Volterra integral equations of the first kind:
\begin{align}\label{VIE-double-A}
\mathbb{P}_{t,b_1(t)}\!\bigl(b_1(T)<X_T<b_2(T)\bigr)
=&-\frac{1}{2}\int_t^T f_1(u)\,p(u,b_1(u);t,b_1(t))\,du \\
&+\frac{1}{2}\int_t^T f_2(u)\,p(u,b_2(u);t,b_1(t))\,du\nonumber\\
\label{VIE-double-A-2}\mathbb{P}_{t,b_2(t)}\!\bigl(b_1(T)<X_T<b_2(T)\bigr)
=&-\frac{1}{2}\int_t^T f_1(u)\,p(u,b_1(u);t,b_2(t))\,du \\
&+\frac{1}{2}\int_t^T f_2(u)\,p(u,b_2(u);t,b_2(t))\,du\nonumber
\end{align}
for $t\in[0,T)$. As before, the behavior at the terminal time $t=T$ is understood in the left-hand limit as $t\to T-$. We first solve the system \eqref{VIE-double-A}-\eqref{VIE-double-A-2} numerically for $f_1$ and $f_2$. Once these functions have been determined, we can compute $G(t,x)$ from \eqref{F-double-A}. An analogous result to Theorem~\ref{cdf-single-const} also holds for the double-barrier case in the time-homogeneous setting.

\begin{remark}
    As in the previous section, from the backward hitting probability $G$, one can obtain the CDF of the first-hitting time, i.e., $F(t,x):=\mathbb{P}_{0,x}(\tau(0)\le t)$ if $b_1$ and $b_2$ are constant thresholds and $X$ is a time-homogeneous diffusion. In this case, we have that $F(t,x)=G(T-t,x)$ for $t\in[0,T]$.
\end{remark}

\begin{example}
\textit{CEV process.}
Suppose that the process $X$ follows
\begin{equation}
dX_t
=
\mu X_t\,dt
+
\sigma X_t^{\beta/2}\,dW_t,
\end{equation}
where $\mu \ne 0$, $\sigma>0$, and $\beta>0$ are constant parameters. If $\beta=2$, then $X$ is a geometric Brownian motion. This model has been used extensively in option pricing, as it allows one to incorporate the so-called leverage effect when $\beta<2$. For $\beta\in(0,2)$, the corresponding transition density is given in \cite{S-1989} by
\begin{equation}\label{cev-density}
p_{\mathrm{CEV}}(s,y;t,x)
=
(2-\beta)\,
\lambda^{\frac{1}{2-\beta}}
(uv^{1-2\beta})^{\frac{1}{4-2\beta}}
e^{-u-v}
I_{\frac{1}{2-\beta}}\!\left(2\sqrt{uv}\right),
\end{equation}
where
\begin{align}
\lambda
&=
\frac{2\mu}
{\sigma^2(2-\beta)\bigl(e^{\mu(2-\beta)(s-t)}-1\bigr)},
\\
u
&=
\lambda x^{2-\beta}e^{\mu(2-\beta)(s-t)},
\qquad
v
=
\lambda y^{2-\beta}.
\end{align}
Figure~\ref{fig:CEV} illustrates the solution for this example,  i.e., the CDF $F$ and the density function $F_t$.
\end{example}

\begin{figure}[h]
    \centering
    \begin{subfigure}[t]{0.48\textwidth}
        \centering
        \includegraphics[width=\textwidth]{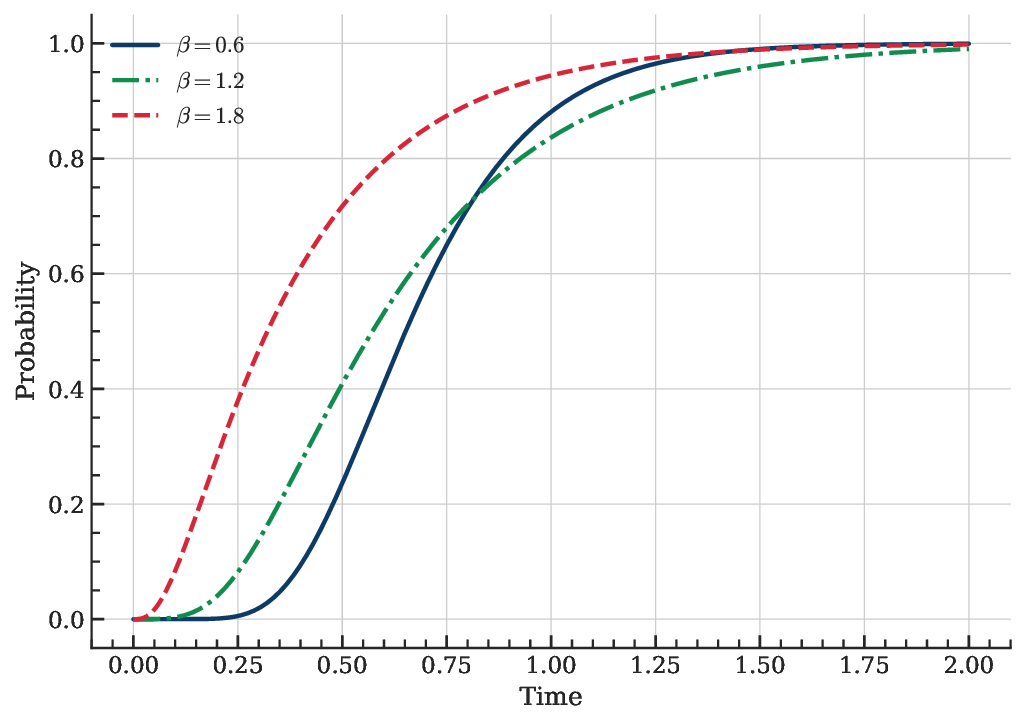}
        \caption{CDF}
        \label{fig:CEV-hitting-cdf}
    \end{subfigure}
    \hfill
    \begin{subfigure}[t]{0.48\textwidth}
        \centering
        \includegraphics[width=\textwidth]{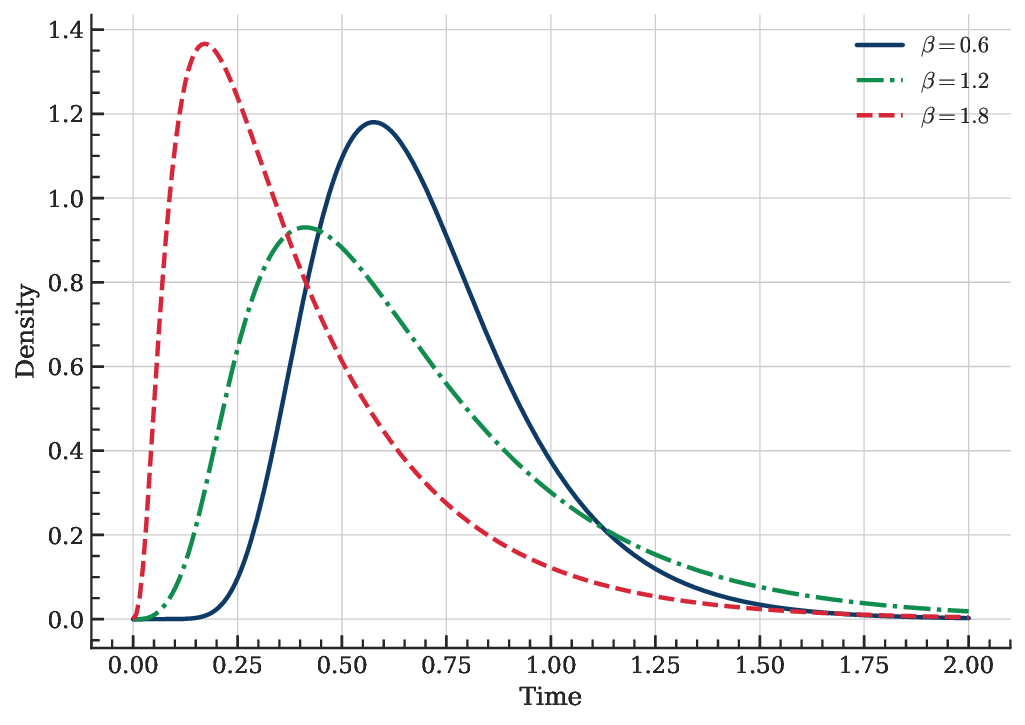}
        \caption{PDF}
        \label{fig:CEV-hitting-pdf}
    \end{subfigure}
    \caption{
    CEV process:
    cumulative distribution function $t\mapsto F(t,X_0)$ (left) and probability density function $t\mapsto F_t(t,X_0)$ (right) of the hitting time of two constant barriers $b_1<b_2$.
    Effect of the elasticity parameter $\beta$. Parameters:  $T=2$, $X_0=5$, $b_1=4$, $b_2=7$, $\mu = 0.5$, $\sigma = 0.5$. The number of time steps used is $m=2^8$.}
    \label{fig:CEV}
    
\end{figure}

\begin{remark}
    
The same methodology also applies to other first-passage problems involving two barriers. As an example, consider the probability that the lower barrier is reached before both the upper barrier and time $T$,
\begin{equation}
G(t,x)=\mathbb{P}_{t,x}(\tau_{b_1}\le \tau_{b_2}\wedge T),
\end{equation}
where $\tau_{b_i}:=\inf\{s\ge t:X_s = b_i(s)\}$ and $i=1,2$, which satisfy the backward problem
\begin{alignat}{2}
&G_t(t,x)+\mathbb{L}_XG(t,x)=0,\qquad&& t\in[0,T),\; b_1(t)<x<b_2(t),\\
&G(t,x)=1,&& t\in[0,T],\; x\le b_1(t),\\
&G(t,x)=0,&& t\in[0,T],\; x\ge b_2(t)\\
&G(T,x)=0,&& b_1(T)<x<b_2(T).
\end{alignat}
Compared with the previous double-barrier formulation, only the terminal and boundary conditions are modified. Therefore, applying the same procedure, we again obtain a coupled system of VIEs of the first kind for the unknown boundary fluxes $f_1$ and $f_2$. By combining the cases above, one can also solve for other variants of the FPT, e.g.,
$
\mathbb{P}_{t,x}(\tau_{b_1}\le T\,|\,\tau_{b_1}\le \tau_{b_2})
=
\frac{\mathbb{P}_{t,x}(\tau_{b_1}\le\tau_{b_2}\wedge T)}
{\mathbb{P}_{t,x}(\tau_{b_1} \le \tau_{b_2})}.
$
\end{remark}

\section{Theoretical Results and Numerical Analysis}
\label{S:properties}

In this section, we provide theoretical characteristics for the obtained equations and give details on the numerical implementation. To begin, we establish the singularity property. To do so, we recall the short-time asymptotic behavior of the transition density, which is locally Gaussian as the time increment tends to zero (see, e.g., \cite{F-1964}).
\begin{lemma}\label{lem-1}
For any function $b:[0,T]\to\mathbb{R}$ in $C^1$, the transition density function $p(s,y;t,x)$ with $0\le t<s\le T$ and $x,y\in\mathbb{R}$ has the following asymptotic behavior:
\begin{equation}
p(s,b(s);t,b(t))
\sim
\frac{1}{\sqrt{2\pi}\,\left|\sigma(t,b(t))\right|}(s-t)^{-1/2}
\quad \text{as } s\to t+,
\end{equation}
and $p(s,y;t,x)\to0$ for $x\neq y$ as $s\to t+$.
\end{lemma}

\begin{example}
We compute the limit $\lim_{s\to t+}(s-t)^{1/2} p(s,x;t,x)$ for the CEV process. Recall the long-term behavior of the modified Bessel function $I_\nu(\cdot)$:
\begin{equation}
I_\nu(z)
\sim
\frac{e^{z}}{\sqrt{2\pi z}},
\quad z\to\infty.
\end{equation}
Setting $y=x$ in \eqref{cev-density}, we obtain
\begin{equation}
e^{-u-v}I_\nu(2\sqrt{uv})
\sim
\frac{e^{-(\sqrt{u}-\sqrt{v})^2}}
{\sqrt{4\pi\sqrt{uv}}}
\sim
\frac{1}{\sqrt{4\pi\lambda x^{2-\beta}}},
\end{equation}
since $\lim_{s \to t+}(\sqrt{u}-\sqrt{v})^2 = 0$. Furthermore,
\begin{equation}
(2-\beta)\,\lambda^{\frac{1}{2-\beta}}
(uv^{1-2\beta})^{\frac{1}{4-2\beta}}
\sim
(2-\beta)\,\lambda\,x^{1-\beta},
\qquad
\lambda
\sim
\frac{2}{\sigma^2(2-\beta)^2(s-t)}.
\end{equation}
Combining the above asymptotics yields
\begin{equation}
\lim_{s\to t+}(s-t)^{1/2} p(s,x;t,x)
=\lim_{s\to t+}
(s-t)^{1/2}
\frac{(2-\beta)\sqrt{\lambda}}
{\sqrt{4\pi}\,x^{\beta/2}}
=\frac{1}{\sqrt{2\pi}\,\sigma x^{\beta/2}}.
\end{equation}
This illustrates the result of Lemma~\ref{lem-1} in the present setting.
\end{example}
As a consequence, we observe that the VIE \eqref{volterra-single} of the first kind exhibits a weak singularity of order $1/2$. It is well known that this type of equation is ill-posed (Section~1.4.2 in \cite{B-2017}) in the sense of Hadamard in $C[0,T]$, even when a unique solution exists. Therefore, our goal is to establish existence and uniqueness in an appropriate function space where the problem is well-posed. We start with the single-barrier case and rewrite equation \eqref{volterra-single} in the general form
\begin{equation}\label{VIE-1-standard}
g(t)
=
\int_0^t
\frac{f(T-u)\,k(t,u)}
     {(t-u)^{1/2}}
\,du,
\qquad
0\le t\le T.
\end{equation}
Here, the left-hand side is
\begin{equation}
g(t):=\mathbb{P}_{T-t,b(T-t)}\!\left(X_T>b(T)\right),
\qquad 0<t\le T,
\end{equation}
and the non-singular kernel is
\begin{equation}
k(t,u):=-\frac{1}{2}(t-u)^{1/2}
p\!\left(T-u,b(T-u);T-t,b(T-t)\right).
\end{equation}
At the endpoint $t=0$, the value of $g$ is understood as the right-hand
limit, namely, $g(0):=\lim_{t\to0+}g(t)$. In particular, $g(0) \ne 0$ and it should not be obtained by directly substituting
$t=0$ into the probabilistic expression, since this would give $\mathbb{P}_{T,b(T)}\!\left(X_T>b(T)\right)=0$,
which is not the value relevant to our equation.
Now, we refer to Theorem~5.4 of \cite{L-1985} or to the original result in \cite{A-1974} for the existence and uniqueness conditions. The core requirements are the smoothness of the data $g(t)$ and $k(t,u)$, which are satisfied under our initial assumptions, together with the condition $k(t,t)\neq0$. The latter follows directly from Lemma~\ref{lem-1}. If these conditions are fulfilled, then the solution belongs to a weighted continuous function space defined by
$\mathcal{C}_\gamma[0,T]
:=
\left\{
f:[0,T]\to\mathbb{R}
:
t^\gamma f(T-t)\in C[0,T]
\right\}$ and $
\gamma>-1$,
with the intrinsic norm
\begin{equation}
\|f\|_\gamma
:=
\max_{0\le t\le T}
\bigl|t^\gamma f(T-t)\bigr|.
\end{equation}
The map
$\mathcal M_\gamma:\mathcal C_\gamma[0,T]\to C[0,T]$,
defined by
$\mathcal M_\gamma(f)(t)=t^{\gamma} f(T-t)$,
is a linear isometric isomorphism.
Hence, $\mathcal{C}_\gamma[0,T]$ is a separable Banach space. The main obstacle in numerical approaches is that $g(0)\neq0$; therefore, the solution $f(T-t)$ blows up at the origin, and standard methods such as those in \cite{L-1985} or \cite{W-1972} are not directly applicable. To overcome this difficulty, we employ a regularization technique. The recent paper by \cite{WNG-2023} presents a finite-term psi-series expansion together with numerical treatments based on Chebyshev collocation methods, showing good numerical performance. In contrast, we propose a simpler approach that preserves efficiency and can be generalized to systems of VIEs. We introduce the regularization function
$\widetilde f(t)=t^\alpha f(T-t)$,
$\alpha\in[1/2,1)$,
which remains bounded at $0$. Thus, equation \eqref{VIE-1-standard} becomes
\begin{equation}\label{VIE-1-standard1}
g(t)
=
\int_0^t
\frac{\widetilde f(u)\,k(t,u)}
     {u^\alpha (t-u)^{1/2}}
\,du,
\qquad
0\le t\le T.
\end{equation}
Next, to solve the equation, we utilize the product integration approach described in \cite{L-1985}. We introduce the uniform grid $\{t_i=i\Delta t,\; i=0,1,\ldots,m\}$
with $\Delta t=T/m$ for some $m$.  The approximate solution at the grid point $t_i$ will be denoted by $\widehat f_i$, $i=0,\ldots,m$. Also, let $g_i=g(t_i)$ and $k_{i,j}=k(t_i,t_j)$.
We approximate $\widetilde f(u)k(t_i,u)$ by a piecewise linear function on each subinterval
\begin{equation}\label{linear-interpolation}
Q_j(t_i,u)
=
\frac{
(t_{j}-u)\,k_{i,j-1} \widehat f_{j-1}
+
(u-t_{j-1})\,k_{i,j}\widehat f_{j}
}
{\Delta t},
\quad
t_{j-1}<u\le t_{j},
\quad
j=1,2,\ldots,i.
\end{equation}
As a result, the integral equation \eqref{VIE-1-standard1} is approximated by the following lower-triangular system for $\widehat f_i$
\begin{equation}\label{VIE-discrete}
\Delta t^{\alpha-1/2}g_i
=
k_{i,0}W_{i,0}^{(1)}\widehat f_0
+
\sum_{j=1}^{i-1}
k_{i,j}
\bigl(
W_{i,j}^{(1)}
+
W_{i,j-1}^{(2)}
\bigr)
\widehat f_j
+
k_{i,i}W_{i,i-1}^{(2)}\widehat f_i,
\quad
i=1,\ldots,m,
\end{equation}
where $W_{i,j}^{(1)}$ and $W_{i,j}^{(2)}$ are precomputed weights
\begin{align}
\label{E:W1}
W_{i,j}^{(1)}
&=
\int_j^{j+1}
(j+1-u)\,u^{-\alpha}(i-u)^{-1/2}
\,du,
\\
\label{E:W2}
W_{i,j}^{(2)}
&=
\int_j^{j+1}
(u-j)\,u^{-\alpha}(i-u)^{-1/2}
\,du,
\end{align}
for $j=0, \ldots, i-1$. The starting value $\widehat f_0$ is determined from the asymptotic behavior of the integral equation as $t\to0+$
\begin{align}
\widehat f_0
=
\begin{cases}
\dfrac{g(0)}{\pi\,k(0,0)}, & \alpha=1/2, \\
0, & \alpha\in(1/2,1).
\end{cases}
\end{align}
At this point, given the numerical solution $\widehat{f}$, we can compute the probability function $G(t,x)$ by performing the numerical integration in \eqref{G-single} using the same quadrature formula on the grid associated with the discrete solution $\widehat f$. In the time-homogeneous setting with a constant barrier, the
first-passage time density can be calculated by applying an explicit finite-difference scheme. 

The next theorem establishes a convergence order of at least $3/2$ for the proposed numerical scheme. This differs from the classical analysis in \cite{W-1972}, where second-order convergence is obtained because $g(0)=0$ and no regularization is needed. Nevertheless, the numerical experiment reported below indicates that the second-order convergence is preserved in practice, suggesting that the reduced theoretical order is an artifact of the proof technique. The sketch of the proof is given in Appendix \ref{S:APP}.

\begin{theorem}
\label{T:error}
Assume the kernel $k(t,u)$ and the solution $\widetilde f(t)$ of \eqref{VIE-1-standard1} are 
\begin{enumerate}
    \item twice continuously differentiable;
    \item derivatives $\widetilde f^{(3)}$, $\partial^3 k/\partial u^3$, and $\partial^3 k/\partial t \partial u^2$ are bounded on the domain $0 \le u \le t \le T$.
\end{enumerate}
Then, the error of the product trapezoidal rule \eqref{VIE-discrete} with $\alpha = 1/2$ for the VIE \eqref{VIE-1-standard1} satisfies
\begin{equation}
\max_{0 \le i \le m} |\widetilde f(t_i) - \widehat f_i| = \mathcal O\left(\Delta t^{3/2} \right).
\end{equation}
\end{theorem}
Now let us move on to the case of two barriers and the corresponding systems. The theoretical results from the one-dimensional case naturally extend to the multi-dimensional setting; see \cite{A-1974}. As before, smoothness conditions on the parameters are required for the existence of a unique solution. We encounter the same numerical issue of an unbounded solution near $0$ for both time-reversed functions $f_1$ and $f_2$. To regularize the functions, we introduce the transformed functions
$\widetilde f_i(t)=t^\alpha f_i(T-t)$,
ensuring that they remain bounded at the origin, $i=1,2$. The modified system is then treated using the product integration method. By a change of variables, it can be rewritten in the following general form
\begin{align}
g_1(t)
&=
-\int_0^t
\frac{\widetilde f_1(u)\,k_{11}(t,u)}
     {u^\alpha(t-u)^{1/2}}
\,du
+
\int_0^t
\frac{\widetilde f_2(u)\,k_{12}(t,u)}
     {u^\alpha(t-u)^{1/2}}
\,du,
\\
g_2(t)
&=
-\int_0^t
\frac{\widetilde f_1(u)\,k_{21}(t,u)}
     {u^\alpha(t-u)^{1/2}}
\,du
+
\int_0^t
\frac{\widetilde f_2(u)\,k_{22}(t,u)}
     {u^\alpha(t-u)^{1/2}}
\,du.
\end{align}
Here, $g_l(t)$ denotes the left-hand sides and $k_{lp}(t,s)$ are the kernels determined from \eqref{VIE-double-A}, for $l,p=1,2$. As before, we use a uniform time grid $\{t_i\}_{i=0}^m$. Functions indexed by superscripts $i,j$ without arguments denote their discrete values at the corresponding grid points $t_i$ and $t_j$. The approximate solutions on the grid are denoted by $\widehat f_1^i$ and $\widehat f_2^i$. For simplicity, we replace the integrands $\widetilde f_p(u)k_{lp}(t,u)$ by the right-endpoint piecewise constant approximations
\begin{equation}
\widehat f_p^{j+1}k_{lp}^{i,j+1},
\quad
t_j<u\le t_{j+1},
\quad
j=0,1,\ldots,i-1.
\end{equation}
This yields the following block lower-triangular system for the approximations $\widehat f_1^j$ and $\widehat f_2^j$
\begin{align}
\begin{cases}
-\displaystyle\sum_{j=1}^{i}
\widehat f_1^j\,k_{11}^{i,j}W_{i,j}
+
\displaystyle\sum_{j=1}^{i}
\widehat f_2^j\,k_{12}^{i,j}W_{i,j}
=
g_1^i,
\\
-\displaystyle\sum_{j=1}^{i}
\widehat f_1^j\,k_{21}^{i,j}W_{i,j}
+
\displaystyle\sum_{j=1}^{i}
\widehat f_2^j\,k_{22}^{i,j}W_{i,j}
=
g_2^i,
\end{cases}
\qquad
i=1,\ldots,m.
\end{align}
Here, $W_{i,j}$ is expressed in terms of the incomplete beta function
\begin{equation}
B_z(a,b)
=
\int_0^z t^{a-1}(1-t)^{b-1}\,dt,
\end{equation}
where $a>0$, $b>0$, and $0\le z\le1$. More precisely,
\begin{equation}
W_{i,j}
=
t_i^{1/2-\alpha}
\left(
B_{\frac{j}{i}}\!\left(1-\alpha,\tfrac12\right)
-
B_{\frac{j-1}{i}}\!\left(1-\alpha,\tfrac12\right)
\right),
\end{equation}
for $j\le i$. The resulting system can be expressed in the compact matrix form
\begin{equation}
\begin{pmatrix}
g_1\\
g_2
\end{pmatrix}
=
\begin{pmatrix}
A_{11} & A_{12}\\
A_{21} & A_{22}
\end{pmatrix}
\begin{pmatrix}
\widehat f_1\\
\widehat f_2
\end{pmatrix}.
\end{equation}
Here,
$g_p=(g_p^1,\ldots,g_p^m)^T$ and
$\widehat f_p=(\widehat f_p^1,\ldots,\widehat f_p^m)^T$
are vectors, while $A_{lp}$, $l,p=1,2$, are lower-triangular matrices. Once the functions $f_1(t)$ and $f_2(t)$ have been computed numerically, the backward hitting probability $G(t,x)$ and its density can be obtained, similarly to the single-barrier case, by applying quadrature rules and finite-difference approximations.

\begin{example}\label{OU-VIE}
Let us consider the following equation with a known solution $f(t)$, obtained as the derivative with respect to $x$ of the hitting-time probability for an OU process at the barrier $b=0$ with certain parameters
\begin{equation}
-\pi^{1/2}
=
\int_0^t
f(u)
\left(
1-e^{-2(t-u)}
\right)^{-1/2}
\,du,
\qquad
f(t)
=
-\frac{2e^{-t/2}}
{\sqrt{2\pi\sinh(t)}},
\end{equation}
for $t\ge0$. Figure~\ref{fig:Error-converg} illustrates the second-order convergence of the numerical solution to the exact solution, when computing the CDF $F(t,x)$. The errors measured in the $l_\infty$ and $l_2$ norms exhibit the expected slope and closely match the reference second-order convergence rate indicated by the red dashed line.
\end{example}
\begin{figure}[h]
    \centering
    \begin{subfigure}[t]{0.48\textwidth}
        \centering
        \includegraphics[width=\textwidth]{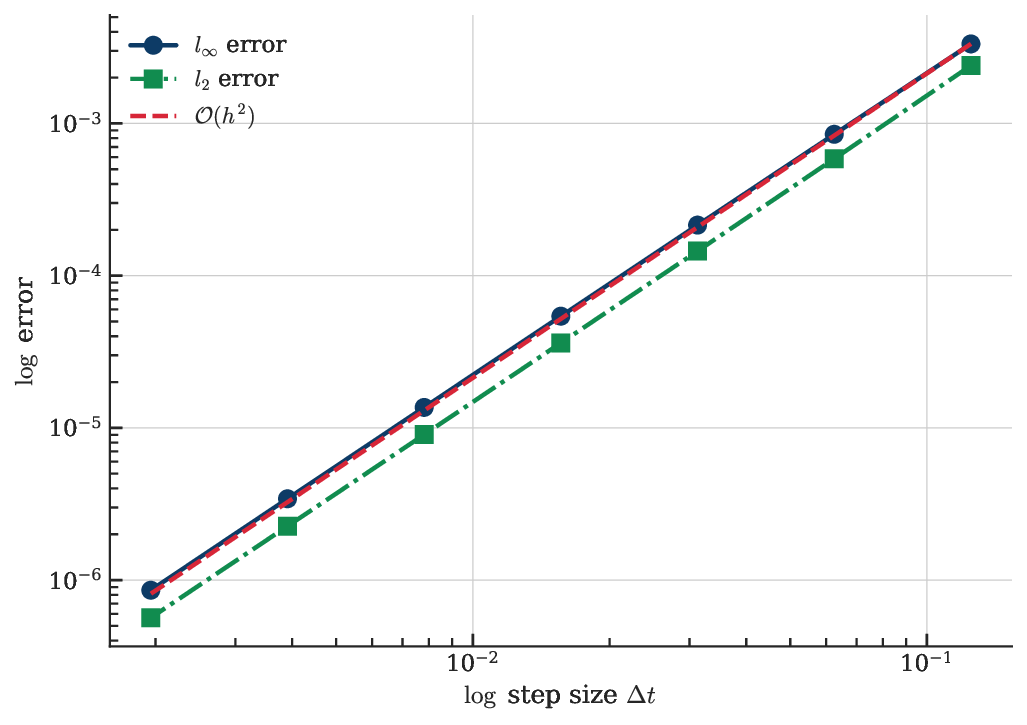}
        \caption{log-log plot}
    \end{subfigure}
    \hfill
    \begin{subfigure}[t]{0.48\textwidth}
        \centering
        \includegraphics[width=\textwidth]{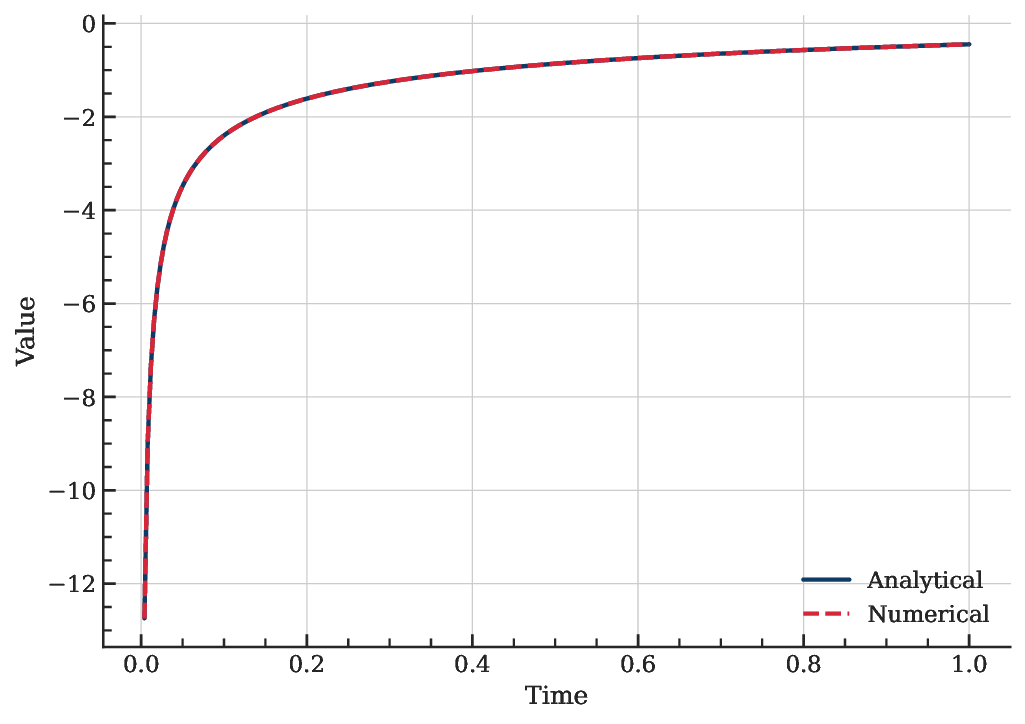}
        \caption{solution}
    \end{subfigure}
    \caption{
       Convergence analysis for Example \ref{OU-VIE}: log-log plot of the error (left) and comparison of numerical and exact solutions for $f(t)$ (right). The number of time steps used in the right panel is $m=2^8$.}\label{fig:Error-converg}
\end{figure}

\section{Derivation of the Second-kind VIE}\label{sec:VIESK} 

As we have seen, recovering $f(t)$ is associated with certain difficulties. In particular, it requires regularization and the numerical solution of a singular first-kind VIE, for which convergence results remain rather limited; see, for example, \cite{WA-1971}, \cite{E-1981}, and the comprehensive discussion in \cite{B-2004}. As noted in Section~6.3 of \cite{B-2004}, this problem ``\textit{poses a formidable challenge, and so far only a few partial results are known}''. Below, we will show that the recovery of $f(t)$ can be reduced to solving a VIE of the second kind with the same weak singularity, for which the numerical theory is much better developed. The idea is to treat the representation \eqref{G-single} as an integral equation with a parameter $x$, which naturally suggests differentiating with respect to this parameter and then evaluating the resulting expression at the barrier $b(t)$, as before. 

To demonstrate this idea, we consider the case of a flat barrier $b$ and a geometric Brownian motion governed by the SDE
\begin{equation}
dX_t = \mu X_t\,dt + \sigma X_t\,dW_t,
\qquad
X_0=x>0,
\end{equation}
where $\mu\in\mathbb{R}$ and $\sigma>0$ are constants. The transition density is given explicitly by
\begin{equation}
p_{GBM}(t,b;u,x)
=
\frac{1}{b\sigma\sqrt{2\pi\theta}}
\exp\!\left(
-\frac{\left(\log(b/x)-(\mu-\sigma^2/2)\theta\right)^2}
{2\sigma^2\theta}
\right),
\label{eq:gbm_density}
\end{equation}
where $\theta=t-u$. Let $\ell=\log\left(x/b\right)$, $r=\mu-\sigma^2/2$.
Differentiating the transition density with respect to the initial point $x$, we obtain
\begin{equation}
\partial_x p_{GBM}(t,b;u,x)
=
-\frac{
\ell+r\theta
}
{b\,x\,\sigma^3\theta^{3/2}\sqrt{2\pi}}
\exp\!\left(
-\frac{\left(\ell+r\theta\right)^2}
{2\sigma^2\theta}
\right).
\end{equation}
We decompose the kernel as
\begin{equation}
\sigma^2b^2\partial_x p_{GBM}(t,b;u,x)
=
-\left(\mathcal K_{\mathrm{sing}}(\theta,x)
+
\mathcal K_{\mathrm{reg}}(\theta,x)\right),
\end{equation}
where
\begin{align}
\mathcal K_{\mathrm{sing}}(\theta,x)
&=
\frac{b}{x}
\frac{\ell}
{\sigma\sqrt{2\pi}\,\theta^{3/2}}
\exp\!\left(
-\frac{(\ell+r\theta)^2}
{2\sigma^2\theta}
\right),
\\
\mathcal K_{\mathrm{reg}}(\theta,x)
&=
\frac{b}{x}
\frac{r}
{\sigma\sqrt{2\pi}\theta^{1/2}}
\exp\!\left(
-\frac{(\ell+r\theta)^2}
{2\sigma^2\theta}
\right).
\end{align}
Now, we rewrite the representation in Theorem~\ref{cdf-single-const} for the hitting probability in our setting and differentiate both sides with respect to $x$. The differentiation under the integral sign is justified by the dominated convergence theorem, since $F_x(u,b+)$ is locally integrable, while the kernel $\partial_x p_{GBM}(t,b;u,x)$ is continuous and admits an integrable majorant (uniformly in $u$ near $t$) for $x >b$. Hence, we have
\begin{equation}
 F_x(t,x)
=
\partial_x\mathbb{P}_{0,x}(X_t\le b)
-\frac12
\int_0^t
F_x(u,b+)\sigma^2b^2
\partial_x p_{GBM}(t,b;u,x)\,du.
\label{eq:dxV_gbm}
\end{equation}
Substituting the transition density function decomposition into \eqref{eq:dxV_gbm}, we obtain
\begin{equation}
F_x(t,x)
=
\partial_x\mathbb{P}_{0,x}(X_t\le b)
+\frac12 \mathcal J_1(t,x)
+\frac12 \mathcal J_2(t,x),
\end{equation}
where
\begin{align}
\mathcal J_1(t,x)
&=
\int_0^t
F_x(u,b+)
\mathcal K_{\mathrm{sing}}(t-u,x)\,du,
\\
\mathcal J_2(t,x)
&=
\int_0^t
F_x(u,b+)
\mathcal K_{\mathrm{reg}}(t-u,x)\,du.
\end{align}
For the singular component, observe that
\begin{equation}
\exp\left(
-\frac{(\ell+r\theta)^2}
{2\sigma^2\theta}
\right)
=
\exp\!\left(
-\frac{\ell^2}
{2\sigma^2\theta}
\right)
\exp\!\left(
-\frac{r\ell}{\sigma^2}
\right)
\exp\!\left(
-\frac{r^2\theta}{2\sigma^2}
\right).
\end{equation}
Hence,
\begin{equation}
\mathcal J_1(t,x)
=
\frac{b}{x}
e^{-r\ell/\sigma^2}
\int_0^t
h_t(\theta)\,
\delta_\ell(\theta)\,
d\theta,
\label{eq:I1}
\end{equation}
where
\begin{equation}
h_t(\theta)
=
\exp\!\left(
-\frac{r^2\theta}{2\sigma^2}
\right)
F_x(t-\theta,b+),
\qquad
\delta_\ell(\theta)
=
\frac{\ell}
{\sigma\sqrt{2\pi}\,\theta^{3/2}}
\exp\!\left(
-\frac{\ell^2}
{2\sigma^2\theta}
\right).
\label{eq:rho}
\end{equation}
The family $\{\delta_\ell\}_{\ell>0}$ is an approximation of the identity on $(0,\infty)$. Indeed, one readily verifies that
\begin{equation}
\delta_\ell(\theta)\ge 0,
\qquad
\int_0^\infty \delta_\ell(\theta)\,d\theta = 1,
\qquad
\lim_{\ell\to0}
\int_\epsilon^\infty
\delta_\ell(\theta)\,d\theta
=0,
\quad \forall\,\epsilon>0.
\end{equation}
Therefore,
\begin{equation}
\lim_{x\to b+}
\mathcal J_1(t,x)
=
\lim_{\ell\to 0+}
e^{-r\ell/\sigma^2}
\int_0^t
h_t(\theta)\delta_\ell(\theta)\,d\theta
=
h_t(0)
=
F_x(t,b+).
\label{eq:jump_gbm}
\end{equation}
Thus, the singular component generates the jump term. Next, consider the regular component. Since $\mathcal K_{\mathrm{reg}}(\theta,x)
=
\mathcal{O}(\theta^{-1/2})$, $
\theta\to0+$,
the kernel is integrable near the origin. By the dominated convergence theorem,
\begin{equation}
\lim_{x\to b+}
\mathcal J_2(t,x)
=
\frac{r}{\sigma\sqrt{2\pi}}
\int_0^t
\frac{F_x(u,b+)}
{(t-u)^{1/2}}
\exp\!\left(-\frac{r^2(t-u)}{2\sigma^2}\right)\,du.
\label{eq:I2_limit}
\end{equation}
Taking the limit $x\to b+$ in \eqref{eq:dxV_gbm}, together with \eqref{eq:jump_gbm} and \eqref{eq:I2_limit}, yields
\begin{align}
F_x(t,b+)
=&
\lim_{x\to b+}\partial_x\mathbb{P}_{0,x}(X_t\le b)
+\frac12F_x(t,b+)
\\
&+\frac{r}{2\sigma\sqrt{2\pi}}
\int_0^t
\frac{F_x(u,b+)}
{(t-u)^{1/2}}
\exp\!\left(-\frac{r^2(t-u)}{2\sigma^2}\right)\,du.\nonumber
\end{align}
Rearranging the terms, we arrive at the following VIE of the second kind with a weak singularity for $F_x(t,b+)$
\begin{equation}
F_x(t,b+)
=
-\frac{2}{b\sigma\sqrt{2\pi t}}
\exp\!\left(-\frac{r^2t}{2\sigma^2}\right)
+
\frac{r}{\sigma\sqrt{2\pi}}
\int_0^t
\frac{F_x(u,b+)}{\sqrt{t-u}}
\exp\!\left(-\frac{r^2(t-u)}{2\sigma^2}\right)\,du.
\label{eq:gbm_second_kind_final}
\end{equation}
The same approach easily extends to BM, GBM, and OU processes with time-dependent coefficients. It should be possible to derive it for the general case \eqref{SDE} as well. The only additional ingredient is an asymptotic expansion of the derivative of the transition density into terms with singularities of orders $(t-u)^{-3/2}$ and $(t-u)^{-1/2}$, together with a regular remainder. Since the latter is generally not available in closed form, the resulting VIE is of the second kind with an unknown kernel, which can nevertheless be computed in particular cases, as illustrated above. Similar arguments extend directly to systems of VIEs in the case of double-barrier problems. We omit the derivations here and leave them for future work. Thereby, we transform the equation into a VIE of the second kind without using the standard method of fractional integration \cite[p.~72]{L-1985}.  Although the issue of unbounded solutions remains, it can be effectively handled using numerical techniques, such as those proposed in \cite{FO-2022}.

\section{Markov Chain Algorithm}
\label{S:MCA}

The general formulas derived in previous sections, e.g., Theorem~\ref{cdf-single}, depend on the transition density function $p(s,y;t,x)$ of the process $X$. In the examples studied above, $p$ had an explicit form. 
However, this is not always the case, or the TDF may have a highly complicated form, as in the CEV model. Therefore, we propose an effective algorithm for the numerical computation of the transition density based on a Markov chain approximation. We follow \cite{RBV-2012} and \cite{BCLP-2018}, but employ the forward Kolmogorov equation.

Primarily, we consider the time-homogeneous case. The key idea is to approximate the process by a continuous-time Markov chain, for which the transition probability matrix between two arbitrary time points can be computed explicitly. The theoretical convergence of the algorithm follows from the Trotter-Kurtz theorem; see \cite{K-1969}. The process $X$ is characterized by
\begin{equation}
dX_t
=
\mu(X_t)\,dt
+
\sigma(X_t)\,dW_t,
\qquad
X_0=x_0,
\end{equation}
with TDF $p(s,y;t,x)$ already defined. Since $X$ is time-homogeneous, its TDF depends only on the elapsed time $s-t$, and we can fix the initial position at $(0,x_0)$; hence, with $t$ denoting the elapsed time,
$p(t,x):=p(t,x;0,x_0)$ satisfies the Kolmogorov forward equation:
\begin{equation}
\partial_t p(t,x)
=
\mathbb L_X^* p(t,x),
\qquad
p(0,x)=\delta(x-x_0),
\end{equation}
where $\delta(x)$ is the Dirac delta function and
$\mathbb{L}_X^*$ is the adjoint generator of $\mathbb L_X$:
\begin{equation}
\mathbb{L}_X^*p(x)
=
-\partial_x\!\bigl[\mu(x)p(x)\bigr]
+
\frac12\,\partial^2_{xx}\!\bigl[\sigma^2(x)p(x)\bigr].
\end{equation}
The diffusion process is approximated by a continuous-time Markov chain with finite state space $\mathcal X = \{x_1,x_2,\ldots,x_n\}$ on a uniform lattice with spacing $\Delta x=x_2-x_1$. The generator $\mathbb A$ is obtained by restricting the differential operator $\mathbb L_X$ to $\mathcal X$ and applying a finite-difference discretization. Let
$
\mathbf P(t)
=
\bigl[p(t,x_i;0,x_j)\bigr]_{i,j=1}^n
$
be the transition density matrix satisfying
\begin{equation}
\frac{d}{dt}\mathbf P(t)
=
\mathbb A^* \mathbf P(t),
\qquad
\mathbf P(0)= \frac{1}{\Delta x}\mathbb I,
\end{equation}
where $\mathbb I$ is the identity matrix and $\mathbb A^* = \mathbb A^{\top}$ is the tridiagonal matrix
\begin{equation}
\label{E:Astar}
\mathbb A^*_{i,j}
=
l_i\delta_{i-1,j}
+
d_i\delta_{i,j}
+
u_i\delta_{i+1,j},
\end{equation}
with $\delta_{p,q}$ the Kronecker delta and coefficients, for $1<i<n$, given by
\begin{equation}\label{E:lidiui}
l_i
=
\frac{\mu(x_{i-1})}{2\Delta x}
+
\frac{\sigma^2(x_{i-1})}{2\Delta x^2},
\qquad
d_i
=
-\frac{\sigma^2(x_i)}{\Delta x^2},
\qquad
u_i
=
-\frac{\mu(x_{i+1})}{2\Delta x}
+
\frac{\sigma^2(x_{i+1})}{2\Delta x^2}.
\end{equation}
When $\mu\not\equiv0$ and $\mu|_{\mathcal X}\neq0$, we assume that the mesh satisfies
\begin{equation}
\Delta x
\le
\inf_{1\le i\le n}
\frac{\sigma^2(x_i)}{|\mu(x_i)|},
\end{equation}
so that the off-diagonal coefficients in \eqref{E:lidiui} are nonnegative and $\mathbb A$ is a valid generator of a Markov chain.
Reflecting (zero-flux) boundary conditions are used by setting the coefficient outside the domain to zero, $l_1=0$ and $u_n=0$, while keeping the off-diagonals at the edges consistent with the interior formula extended to the boundaries:
\begin{equation}
u_1 = -\frac{\mu(x_2)}{2\Delta x} + \frac{\sigma^2(x_2)}{2\Delta x^2},\qquad
l_n = \frac{\mu(x_{n-1})}{2\Delta x} + \frac{\sigma^2(x_{n-1})}{2\Delta x^2}.
\end{equation}
The diagonal elements $d_1$ and $d_n$ are then determined by the requirement that the column sums of $\mathbb A^*$ vanish (mass conservation), which gives
\begin{equation}
d_1 = -l_2 = -\frac{\mu(x_1)}{2\Delta x} - \frac{\sigma^2(x_1)}{2\Delta x^2},\qquad
d_n = -u_{n-1} = \frac{\mu(x_n)}{2\Delta x} - \frac{\sigma^2(x_n)}{2\Delta x^2}.
\end{equation}
For a sufficiently large computational domain, the probability mass near
the artificial boundaries is negligible, so this boundary choice has little
effect on the numerical results. Using operator theory, the solution can be expressed as a matrix exponential:
\begin{equation}
\mathbf P(t)
=
e^{t\mathbb A^*}\mathbf P(0).
\end{equation}
Calculations on the time grid $\{t_i\}_{i=0}^m$ can be performed iteratively according to
\begin{equation}
\mathbf P(t_{i+1})
=
e^{\Delta t\,\mathbb A^*}\mathbf P(t_i).
\end{equation}
Thus, it is sufficient to compute the matrix exponential
$e^{\Delta t\,\mathbb A^*}$
only once. For this purpose, we employ the scaling-and-squaring method combined with the Pad\'e approximation described in \cite{H-2005}. The idea is to exploit the fundamental property of the matrix exponential. For an arbitrary matrix $A\in\mathbb R^{n\times n}$, one writes
$e^A
=
\left(e^{A/2^r}\right)^{2^r}$ 
and chooses an integer $r$ such that the matrix norm $\|A/2^r\|$ is of order one. Then $e^{A/2^r}$ is approximated by a rational function (Pad\'e approximation), and $e^A$ is computed by applying the squaring operation $r$ times.
In the case of a constant barrier, the problem simplifies significantly. It is sufficient to repeatedly multiply a vector $v\in\mathbb R^n$ by a matrix $e^A\in\mathbb R^{n\times n}$, which can be done efficiently using a single Krylov--Arnoldi decomposition from \cite{S-1992}, without explicitly computing $e^A$. The generalization to a piecewise time-homogeneous process is straightforward. If the process $X$ is homogeneous on each interval $[t_i,t_{i+1})$ with the corresponding generator $\mathbb A_i$, then
\begin{equation}
\mathbf P(t_{i+1})
=
e^{\Delta t\,\mathbb A_i^*}\mathbf P(t_i).
\end{equation}
In the general time-inhomogeneous case, the computation can be performed using the Magnus expansion; see \cite{BCOR-2009}, in particular, Section~5 for numerical integration methods.

\begin{example} \label{CIR-example}
\textit{Feller process.}
Assume that $X$ follows
\begin{equation}
dX_t
=
\kappa(\theta-X_t)\,dt
+
\xi\sqrt{X_t}\,dW_t,
\end{equation}
where $\kappa$, $\theta$, and $\xi$ are constant parameters. The law of $X$ is given by a non-central chi-squared distribution, with the well-known transition density
\begin{equation}
p_{\mathrm{Feller}}(s,y;t,x)
=
\lambda e^{-(u+v)}
\left(\frac{v}{u}\right)^{\nu/2}
I_\nu\!\bigl(2\sqrt{uv}\bigr),
\qquad
y\ge0,
\end{equation}
where
\begin{equation}
u
=
\lambda x e^{-\kappa(s-t)},
\qquad
v
=
\lambda y,
\qquad
\nu
=
\frac{d}{2}-1,
\end{equation}
and
\begin{equation}
\lambda
=
\frac{2\kappa}
{\xi^2\bigl(1-e^{-\kappa(s-t)}\bigr)},
\qquad
d
=
\frac{4\kappa\theta}{\xi^2}.
\end{equation}
In this example,
$\mu(x)=\kappa(\theta-x)$
and
$\sigma(x)=\xi\sqrt{x}$.
The approximate transition law is obtained by substituting these coefficients into \eqref{E:Astar}--\eqref{E:lidiui}. Figure~\ref{fig:CIR-main} illustrates that the distribution functions generated by the Markov chain approximation are indistinguishable from the analytical formulas and the Monte Carlo estimates.
\end{example}
\begin{figure}[h!]
    \centering
    \begin{subfigure}[t]{0.48\textwidth}
        \centering
        \includegraphics[width=\textwidth]{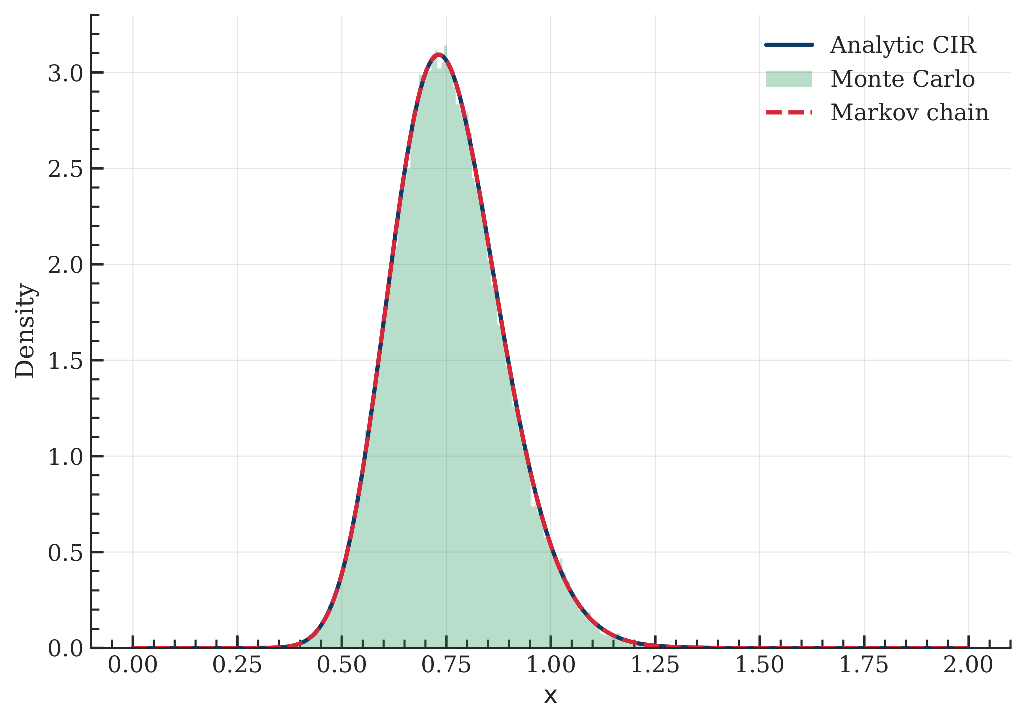}
        \caption{Probability density functions}
        \label{fig:CIR-TDF}
    \end{subfigure}
    \hfill
    \begin{subfigure}[t]{0.48\textwidth}
        \centering
        \includegraphics[width=\textwidth]{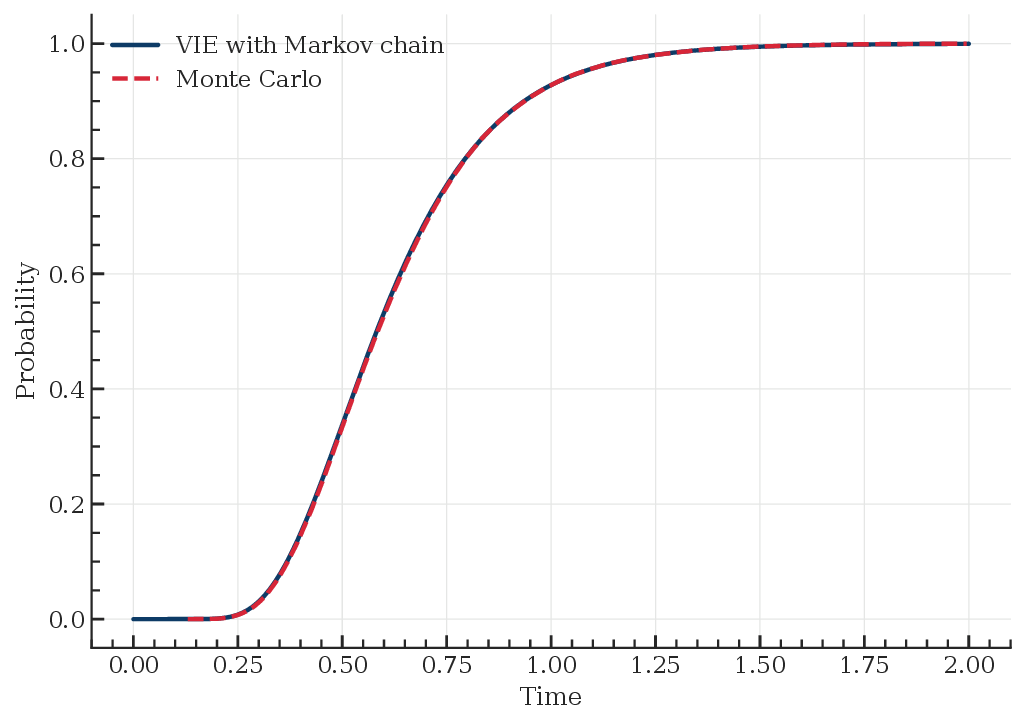}
        \caption{Hitting time distribution}
        \label{fig:CIR-hitting-cdf}
    \end{subfigure}

    \caption{Feller process: TDF $p(T,x;0,b)$ of $X_T$ (left) and cumulative distribution function $t\mapsto F(t,x)$ for $t\in[0,T]$ (right) of the hitting time of a lower constant barrier. Parameters: $T = 2$, $X_0 = 1.75$, $b = 1$, $\kappa = 2$, $\theta = 0.75$, $\xi = 0.3$.  Computations are based on the Monte Carlo method with $10^4$ time steps and $10^5$ trajectories, and the VIE-MCA method with $m = 2^8$ time steps and $n = 2^8$ spatial steps.}
    \label{fig:CIR-main}
\end{figure}

\section{Local Stochastic Volatility} 
\label{S:volatility}

In this section, we extend our approach to a class of stochastic volatility models, often referred to as Local Stochastic Volatility (LSV) models, which are widely used in finance to calibrate the implied volatility surface while preserving realistic volatility dynamics. We consider the system
\begin{align}
dX_t &= \mu(t,X_t)\,dt
      + \sigma(t,X_t)\varphi(V_t)\,dW_t,
      \qquad X_0 = x,\\
dV_t &= \eta(t,V_t)\,dt
      + \xi(t,V_t)\,dB_t,
      \qquad V_0 = v,
\end{align}
where the coefficients $(\mu,\sigma,\varphi,\eta,\xi)$ satisfy suitable regularity conditions ensuring the existence and uniqueness of a global strong solution $(X,V)$; see \cite{LSZ-2020}. The $\mathbb{F}$--Brownian motions $W$ and $B$ under the probability measure $\mathbb{P}$ may be correlated: 
\begin{equation}
dW_t
=
\rho(t)\,dB_t
+
\sqrt{1-\rho^2(t)}\,dW_t^\perp,
\end{equation}
where $\rho(t)\in(-1,1)$ is a time-dependent correlation coefficient and $W^\perp$ is an $\mathbb{F}$--Brownian motion independent of $B$.

As in the previous sections, we are interested in the first-hitting time problem for the process $X$. In the financial context, this problem arises naturally in structural credit risk models and in the pricing of barrier options with one or two barriers.

For implementation, we apply a conditional Monte Carlo method to estimate the probability of hitting the barrier before the terminal time $T$, in the spirit of \cite{W-1997}, \cite{LP-2009}, and \cite{DKS-2025}. To this end, we simulate $N$ trajectories of the Brownian driver $B$ on a uniform grid and use an appropriate discretization scheme, such as Euler or Milstein, to obtain the volatility factor $V$. We denote the resulting trajectories by $B^{(n)}$ and $V^{(n)}$, respectively, for $n=1,\dots,N$.

The decomposition of the Brownian drivers gives the following It\^o dynamics, before any path interpolation
\begin{align}
dX_t^{(n)}
=&
\mu(t,X_t^{(n)})\,dt
+
\rho(t)\sigma(t,X_t^{(n)})\varphi(V_t^{(n)})\,dB_t^{(n)}
\nonumber\\
&+
\sqrt{1-\rho^2(t)}
\,\sigma(t,X_t^{(n)})
\,\varphi(V_t^{(n)})
\,dW_t^\perp,
\end{align}
with $X_0^{(n)}=x$. Fix $n$ and consider the uniform grid $\{t_i\}_{i=0}^m$ introduced above, with time step $\Delta t=T/m$. Set
\begin{equation}
    v_i^{(n)}:=V_{t_i}^{(n)},\qquad
    \rho_i:=\rho(t_i),\qquad
    \Delta B_i^{(n)}
    :=B_{t_{i+1}}^{(n)}-B_{t_i}^{(n)}.
\end{equation}
On each interval $[t_i,t_{i+1})$, we freeze the volatility factor
and the deterministic time arguments at their left-endpoint values and replace $B^{(n)}$ by its piecewise linear interpolation, while retaining the dependence on the state variable. Thus, for
$t\in[t_i,t_{i+1})$, define
\begin{align}
\widetilde\mu^{(n)}(t,x)
& =
\mu(t_i,x)
+\rho_i\sigma(t_i,x)\varphi(v_i^{(n)})
    \frac{\Delta B_i^{(n)}}{\Delta t}-\frac12\rho_i^2\varphi^2(v_i^{(n)})
    \sigma(t_i,x)\partial_x\sigma(t_i,x),
\\
\widetilde\sigma^{(n)}(t,x)
&=
\sqrt{1-\rho_i^2}\,
\sigma(t_i,x)\varphi(v_i^{(n)}).
\end{align}
The last term in $\widetilde\mu^{(n)}$ is the It\^o correction
that arises when the $B$--driven stochastic integral is replaced by its
smooth-path Wong--Zakai approximation; see \cite{WZ-1965}. The integral with respect to $W^\perp$ remains an It\^o integral. Conditional on the simulated trajectory of $B^{(n)}$, and hence on the
discretized path $V^{(n)}$, we define the approximating diffusion, again written as $X^{(n)}$, by
\begin{equation}
dX_t^{(n)}
=
\widetilde \mu^{(n)}(t,X_t^{(n)})\,dt
+
\widetilde \sigma^{(n)}(t,X_t^{(n)})\,dW_t^\perp.
\end{equation}
For each fixed simulated trajectory, the coefficients $\widetilde\mu^{(n)}$ and $\widetilde\sigma^{(n)}$ are deterministic. The preceding backward representations can therefore be applied to this conditional diffusion. Averaging the resulting conditional hitting probabilities yields the estimator for the original model:
\begin{equation}\label{MC-estimate}
G_{\mathrm{LSV}}^T(0,x,v)
\approx
\frac{1}{N}
\sum_{n=1}^{N}
G^T\!\left(0,x;\widetilde\mu^{(n)},\widetilde\sigma^{(n)}\right),
\end{equation}
where $G^T(0,x;\widetilde\mu^{(n)},\widetilde\sigma^{(n)})$ is computed from the appropriate single- or double-barrier representation, such as \eqref{G-single} or \eqref{F-double-A}.

To improve the convergence rate of the estimator, we follow \cite{W-1997} and employ a quasi-Monte Carlo approach. Specifically, low-discrepancy sequences, such as Sobol' or Halton sequences, are used to generate the Brownian driver $B$ of the stochastic volatility process. We choose the number of time steps in the discretization as $m=2^I$ for some integer $I$. For each $n=1,\ldots,N$, let
$
\mathbf{x}^{\,n}
=
(x_1^{\,n},\ldots,x_m^{\,n})
\in [0,1]^m
$
denote the $n$-th point of the low-discrepancy sequence. The components are transformed into standard Gaussian variates through
$
z_j^{\,n}
=
\Phi^{-1}(x_j^{\,n})$, $ j=1,\ldots,m$,
where $\Phi$ denotes the CDF of the standard normal distribution.

The Brownian paths are then generated using the Brownian bridge construction with binary midpoint refinement; see \cite{CM-1995}. We initialize
$
B_0^{(n)}=0$, $
B_T^{(n)}=\sqrt{T}\,z_m^{\,n}$,
and recursively construct the intermediate values. At level $i=1,\ldots,I$, let $h_i=T/2^i$. For each midpoint $(2j-1)h_i$, we set
\begin{equation}
B_{(2j-1)h_i}^{(n)}
=
\frac{
B_{2(j-1)h_i}^{(n)}
+
B_{2jh_i}^{(n)}
}{2}
+
\sqrt{\frac{h_i}{2}}\,
z_{2^{i-1}+j-1}^{\,n},
\qquad
j=1,\ldots,2^{i-1}.
\end{equation}
Combined with quasi-Monte Carlo sampling, the Brownian bridge construction reduces the effective dimension of the simulation problem and typically leads to improved convergence compared with standard Monte Carlo methods.
\begin{example}
One of the most widely used stochastic volatility models is the Heston model. It is defined by the system
\begin{align}
dX_t &= \mu X_t\,dt + \sqrt{V_t}\,X_t\,dW_t,
\qquad X_0=x,
\\
dV_t &= \kappa(\theta - V_t)\,dt + \xi \sqrt{V_t}\,dB_t,
\qquad V_0=v,
\end{align}
where $V$ denotes the instantaneous variance process, modeled as a mean-reverting square-root diffusion (also known as a Feller process). The parameters $(\mu,\kappa,\theta,\xi)$ and the initial values $(X_0,V_0)$ are assumed to be constant, and $W$ with $B$ have a constant correlation $\rho$. Figure~\ref{fig:Heston} illustrates the numerical solution for the CDF $F_{\mathrm{LSV}} (t,x,v):=\mathbb{P}_{0,x,v}(\tau\le t)$ and the PDF, where $\tau=\inf\{s\ge 0:X_s\le b\}$. It highlights the impact of the vol-of-vol parameter $\xi$.
\end{example}
\begin{figure}[h]
    \centering
    \begin{subfigure}[t]{0.48\textwidth}
        \centering
        \includegraphics[width=\textwidth]{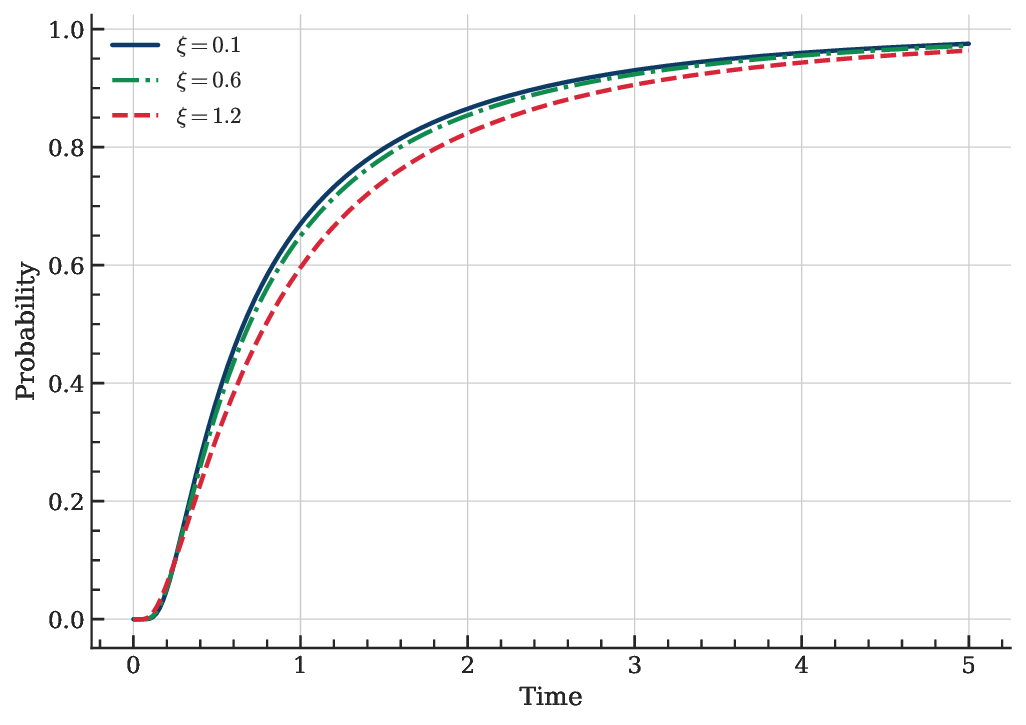}
        \caption{CDF}
        \label{fig:HESTON-hitting-cdf}
    \end{subfigure}
    \hfill
    \begin{subfigure}[t]{0.48\textwidth}
        \centering
        \includegraphics[width=\textwidth]{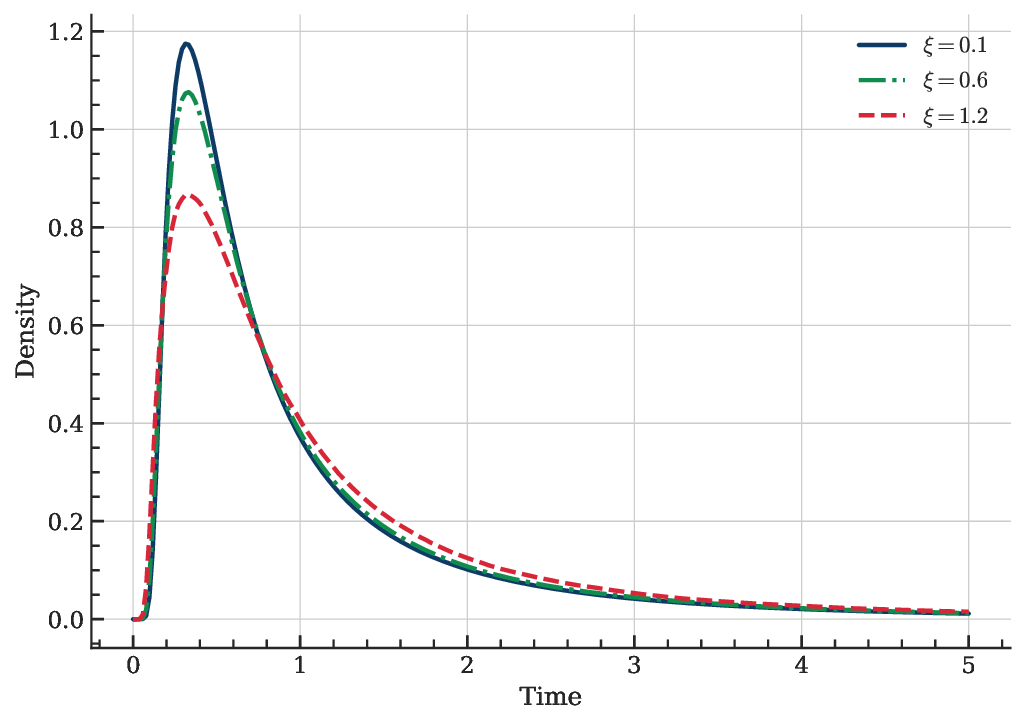}
        \caption{PDF}
        \label{fig:HESTON-hitting-pdf}
    \end{subfigure}
    \caption{
    Heston model:
    cumulative distribution function $t\mapsto F_{\mathrm{LSV}}(t,x,v)$ (left) and probability density function $t\mapsto (F_{\mathrm{LSV}})_t(t,x,v)$ (right) of the first-hitting time for  $t\in[0,T]$. Effect of $\xi$ (vol-of-vol). Parameters:  $T=5$, $X_0=10$, $b=7$, $V_0 = 0.09$,  $\mu=-0.1$, $\theta = 1.21$, $\kappa = 0.5$, $\rho=0$. Computations use the MC-VIE method in Section~\ref{S:volatility} with $m = 2^8$ time steps and $N = 2^{10}$ trajectories.
    } \label{fig:Heston}
\end{figure}

\appendix 

\section*{Appendix}
\section{Proof of Theorem~\ref{T:error}}
\label{S:APP}

The convergence of the trapezoidal method for VIEs (with 
$g(0) = 0$) was established in \cite{E-1981} and \cite{W-1972} using different techniques. The first proof relies on the semicirculant property of the coefficient matrix, which no longer holds in our setting, and a power series representation, while the second is based on a regularity theorem for linear systems (see  Chapter I, \S 2 in \cite{KK-1958}) --- an approach we will also follow.
\begin{lemma}\label{Kantorovich-Krylov}
Let $\{x_i\}_{i\ge 1}$ be a sequence of real numbers. Assume that there exists a constant $K>0$ such that $|x_1|\le K$, and for all $i\ge 1$:
\begin{equation}
|x_{i+1}| \le \sum_{j=1}^{i} |c_{i,j}|\,|x_j| + |b_i|.
\end{equation}
Suppose further that
\begin{equation}
\varrho_i := 1-\sum_{j=1}^{i}|c_{i,j}| > 0,
\quad
|b_i| \le K\varrho_i,
\quad i\ge 1.
\end{equation}
Then $|x_i| \le K$, for all $i \ge 1$.
\end{lemma}
Anticipating the application of Lemma~\ref{Kantorovich-Krylov} below,
$\{x_i\}_{i\ge 1}$ will play the role of the error sequence, while $K$ will be of order
$\mathcal{O}(\Delta t^{3/2})$. We consider the limit $\Delta t \to 0$, and we will omit it hereafter.

For the linear interpolant $Q_j(t_i,u)$ defined in \eqref{linear-interpolation}, where the exact values $\widetilde f$ are used in place of the approximations $\widehat f$, let $R_j(t_i,u) = Q_j(t_i,u) - \widetilde f(u)k(t_i,u)$ denote the interpolation remainder. By the standard remainder formula for linear Lagrange interpolation,
\begin{equation}\label{remainder}
R_j(t_i,u)
=
\frac{1}{2}(t_{j}-u)(u-t_{j-1})\phi(t_i,t_j,u),
\quad
t_{j-1}<u\le t_{j},
\quad
j=1,2,\ldots,i,
\end{equation}
where
\begin{equation}
\phi(t_i,t_j,u)
=
\partial^2_{uu}\!\left[\widetilde f(u)k(t_i,u)\right]\Big|_{u=t_j}
+r_{i,j}(u),
\qquad
r_{i,j}(u)=\mathcal{O}(\Delta t).
\end{equation}
Discretizing \eqref{VIE-1-standard1}, we obtain
\begin{align}
g(t_i)
=&
\sum_{j=1}^{i}
\int_{t_{j-1}}^{t_j}
\frac{Q_j(t_i,u)}
     {u^\alpha (t_i-u)^{1/2}}
\,du
-
\sum_{j=1}^{i}
\int_{t_{j-1}}^{t_j}
\frac{R_j(t_i,u)}
     {u^\alpha (t_i-u)^{1/2}}
\,du
\\
=&\,
\Delta t^{1/2-\alpha}
k_{i,0}W_{i,0}^{(1)}\widetilde f(0)
+
\Delta t^{1/2-\alpha}
\sum_{j=1}^{i-1}
k_{i,j}
\bigl(
W_{i,j}^{(1)}
+
W_{i,j-1}^{(2)}
\bigr)
\widetilde f(t_j)
\nonumber\\
&+
\Delta t^{1/2-\alpha}
k_{i,i}W_{i,i-1}^{(2)}
\widetilde f(t_i)
-
\sum_{j=1}^{i}
\int_{t_{j-1}}^{t_j}
\frac{R_j(t_i,u)}
     {u^\alpha (t_i-u)^{1/2}}
\,du,
\nonumber
\end{align}
where the weights $W_{i,j}^{(1)}$ and $W_{i,j}^{(2)}$ are defined in \eqref{E:W1} and \eqref{E:W2}, respectively. Subtracting \eqref{VIE-discrete} from the above identity, we find that the error $\varepsilon_i = \widetilde f(t_i) - \widehat f_i$ satisfies the recursion
\begin{align}
\Delta t^{1/2-\alpha}
k_{i,i}W_{i,i-1}^{(2)}
\varepsilon_i
+
\Delta t^{1/2-\alpha}
\sum_{j=1}^{i-1}
k_{i,j}
\bigl(
W_{i,j}^{(1)}
+
W_{i,j-1}^{(2)}
\bigr)
\varepsilon_j
&\\
=
\sum_{j=1}^{i}
\int_{t_{j-1}}^{t_j}
\frac{R_j(t_i,u)}
     {u^\alpha (t_i-u)^{1/2}}
\,du
-
\Delta t^{1/2-\alpha}
k_{i,0}W_{i,0}^{(1)}
\varepsilon_0.
\nonumber
\end{align}
Finally, subtracting the above relation from the corresponding relation with $i$ replaced by $i+1$ yields the error equation
\begin{align}\label{error-equation0}
k_{i+1,i+1}W_{i+1,i}^{(2)}\varepsilon_{i+1}
&+
\Big(
k_{i+1,i}\bigl(W_{i+1,i}^{(1)}+W_{i+1,i-1}^{(2)}\bigr)
-k_{i,i}W_{i,i-1}^{(2)}
\Big)\varepsilon_i
\\
&+
\sum_{j=1}^{i-1}
\Big[
k_{i+1,j}\bigl(W_{i+1,j}^{(1)}+W_{i+1,j-1}^{(2)}\bigr)
-k_{i,j}\bigl(W_{i,j}^{(1)}+W_{i,j-1}^{(2)}\bigr)
\Big]
\varepsilon_j
\nonumber\\
=&
\Delta t^{\alpha-1/2}
\sum_{j=1}^{i+1}
\int_{t_{j-1}}^{t_j}
\frac{R_j(t_{i+1},u)}
     {u^\alpha (t_{i+1}-u)^{1/2}}
\,du
\nonumber\\
&-\Delta t^{\alpha-1/2}
\sum_{j=1}^{i}
\int_{t_{j-1}}^{t_j}
\frac{R_j(t_i,u)}
     {u^\alpha (t_i-u)^{1/2}}
\,du
\nonumber\\
&
-k_{i+1,0}W_{i+1,0}^{(1)}\varepsilon_0
+k_{i,0}W_{i,0}^{(1)}\varepsilon_0.
\nonumber
\end{align}
For the remainder of the proof, we first consider the case $k(t,u)\equiv1$ and return to the general kernel at the end. Moreover, the most relevant case for our purposes is $\alpha=1/2$. Therefore, throughout the proof, we restrict ourselves to this case. The error equation then simplifies to
\begin{align}\label{error-equation}
W_{i+1,i}^{(2)}&\varepsilon_{i+1}
+
\Big(
W_{i+1,i}^{(1)}
+
W_{i+1,i-1}^{(2)}
-
W_{i,i-1}^{(2)}
\Big)
\varepsilon_i
\\
+& \nonumber
\sum_{j=1}^{i-1}
\Big[
\bigl(W_{i+1,j}^{(1)}+W_{i+1,j-1}^{(2)}\bigr)
-
\bigl(W_{i,j}^{(1)}+W_{i,j-1}^{(2)}\bigr)
\Big]
\varepsilon_j
=
\mathcal S_{i+1}-\mathcal S_i, 
\end{align}
where
\begin{equation}
\mathcal S_i
=
\sum_{j=1}^{i}
\int_{t_{j-1}}^{t_j}
\frac{R_j(t_i,u)}
     {u^{1/2}(t_i-u)^{1/2}}
\,du
-
W_{i,0}^{(1)}\varepsilon_0.
\end{equation}
Dividing both sides by $W_{i+1,i}^{(2)}$, we obtain
\begin{equation}
\varepsilon_{i+1}
=
\sum_{j=1}^{i} c_{i,j}\varepsilon_j
+b_i,
\end{equation}
where
\begin{align}
c_{i,j}
&=
\frac{
\bigl(W_{i,j}^{(1)}+W_{i,j-1}^{(2)}\bigr)
-
\bigl(W_{i+1,j}^{(1)}+W_{i+1,j-1}^{(2)}\bigr)
}
{W_{i+1,i}^{(2)}},
\qquad j<i,
\\
c_{i,i}
&=
\frac{
W_{i,i-1}^{(2)}
-
\bigl(W_{i+1,i}^{(1)}+W_{i+1,i-1}^{(2)}\bigr)
}
{W_{i+1,i}^{(2)}}, \qquad
b_i =
\frac{\mathcal S_{i+1}-\mathcal S_i}{W_{i+1,i}^{(2)}}.
\end{align}
We will find $\varrho_i$ and show that, for $\alpha = 1/2$, it behaves as $\mathcal{O}(i^{-1})$ as $i \to \infty$, which is implicitly assumed hereafter. First, note that the coefficients $c_{i,j}$ are positive for all $1\le j\le i$. Using the definitions of $c_{i,j}$, we obtain
\begin{align}
    W_{i+1,i}^{(2)}& \left(1 - \sum_{j=1}^i |c_{i,j}| \right) \\
    =&W_{i+1,i}^{(2)} - \sum_{j=1}^{i-1} \left[ \left(W_{i,j}^{(1)} + W_{i,j-1}^{(2)}\right) - \left(W_{i+1,j}^{(1)} + W_{i+1,j-1}^{(2)}\right) \right]  \nonumber\\
    &-W_{i,i-1}^{(2)} + \left(W_{i+1,i}^{(1)} + W_{i+1,i-1}^{(2)}\right) \nonumber\\
    =& -\sum_{j=1}^{i-1}  \left(W_{i,j}^{(1)} + W_{i,j}^{(2)}\right) + \sum_{j=1}^{i}  \left(W_{i+1,j}^{(1)} + W_{i+1,j}^{(2)}\right)\nonumber  \\
    & +W_{i+1,0}^{(2)} - W_{i,0}^{(2)}\nonumber \\
    =& \underbrace{
-\int_1^i
u^{-1/2}(i-u)^{-1/2}\,du
+
\int_1^{i+1}
u^{-1/2}(i+1-u)^{-1/2}\,du
}_{\mathcal I_1(i)}
\nonumber\\
    &+
\underbrace{
\int_0^1
u^{1/2}(i+1-u)^{-1/2}\,du
-
\int_0^1
u^{1/2}(i-u)^{-1/2}\,du
}_{\mathcal I_2(i)}.
\nonumber
\end{align}
The above integrals can be evaluated explicitly:
\begin{align}
\int_1^k
u^{-1/2}(k-u)^{-1/2}\,du
&=
\pi
-
2\arcsin\!\left(\frac{1}{\sqrt{k}}\right),
\\
\int_0^1
u^{1/2}(k-u)^{-1/2}\,du
&=
k\,\arctan\!\left(\frac{1}{\sqrt{k-1}}\right)
-\sqrt{k-1}.
\end{align}
In particular, it follows that $\varrho_i>0$. Moreover, a Taylor expansion yields
\begin{equation}
\mathcal I_1(i)
=
i^{-3/2}
+
\mathcal{O}(i^{-5/2}),
\qquad
\mathcal I_2(i)
=
-\frac13\,i^{-3/2}
+
\mathcal{O}(i^{-5/2}).
\end{equation}
It remains to note that $W_{i+1,i}^{(2)} = \mathcal{O}(i^{-1/2})$, which yields the desired expression
\begin{equation}
0<\varrho_i
=
1-\sum_{j=1}^{i}|c_{i,j}|
=
\mathcal{O}(i^{-1}).
\end{equation}
We now turn to the coefficients $b_i$. We want to show that $|b_i|=\mathcal{O}(\Delta t^{3/2}i^{-1})$. To this end, we expand the difference $\mathcal S_{i+1}-\mathcal S_i$. Since $k(t,u)\equiv1$, the remainder term \eqref{remainder} involves only $\widetilde f^{(2)}$, and after substituting \eqref{remainder} and performing a change of variables we obtain
\begin{align}
\mathcal S_{i+1}-&\mathcal S_i
\\
=&\frac12
\sum_{j=1}^{i+1}
\widetilde f^{(2)}(t_{j-1})
\int_{t_{j-1}}^{t_j}
\frac{(t_j-u)(u-t_{j-1})}
     {u^{1/2}(t_{i+1}-u)^{1/2}}
\,du
\nonumber\\
&-
\frac12
\sum_{j=1}^{i}
\widetilde f^{(2)}(t_{j-1})
\int_{t_{j-1}}^{t_j}
\frac{(t_j-u)(u-t_{j-1})}
     {u^{1/2}(t_i-u)^{1/2}}
\,du
\nonumber\\
&+
W_{i,0}^{(1)}\varepsilon_0
-
W_{i+1,0}^{(1)}\varepsilon_0
+
\omega_i(\Delta t)
\nonumber\\
=&
\bigl(
W_{i,0}^{(1)}
-
W_{i+1,0}^{(1)}
\bigr)
\varepsilon_0
\nonumber\\
&+
\frac{\Delta t^2}{2}
\widetilde f^{(2)}(t_i)
\int_0^1
\frac{u(1-u)}
     {(i+u)^{1/2}(1-u)^{1/2}}
\,du
\nonumber\\
&+
\frac{\Delta t^2}{2}
\sum_{j=1}^{i}
\widetilde f^{(2)}(t_{j-1})
\int_0^1
\frac{u(1-u)}
     {(j-1+u)^{1/2}}
\Bigl(
(i-j+2-u)^{-1/2}
-
(i-j+1-u)^{-1/2}
\Bigr)
\,du
\nonumber\\
&+
\omega_i(\Delta t),
\nonumber
\end{align}
where $\omega_i(\Delta t)$ collects the higher-order interpolation-remainder terms and satisfies $\omega_i(\Delta t) = \mathcal{O}(\Delta t^3 i^{-3/2})$. Furthermore, one readily verifies that
\begin{equation}
\sum_{j=1}^{i}
\frac{
(i-j+2-u)^{-1/2}
-
(i-j+1-u)^{-1/2}
}
{(j-1+u)^{1/2}}
=
-(i+u)^{-1/2}(1-u)^{-1/2}
+
\mathcal{O}(i^{-3/2}).
\end{equation}
Interchanging the order of summation and integration (which is justified by the Lebesgue dominated convergence theorem), we obtain
\begin{align}
\left|
\psi_i-\sum_{j=1}^{i}\chi_{i,j}
\right|
&=
\mathcal{O}(i^{-3/2}),
\\
\psi_i
&=
\int_0^1
u(1-u)
(i+u)^{-1/2}
(1-u)^{-1/2}
\,du,
\nonumber\\
\chi_{i,j}
&=
-\int_0^1
u(1-u)
\frac{
(i-j+2-u)^{-1/2}
-
(i-j+1-u)^{-1/2}
}
{(j-1+u)^{1/2}}
\,du .
\nonumber
\end{align}
Applying Abel's transformation yields
\begin{equation}
\sum_{j=1}^{i}
\widetilde f^{(2)}(t_{j-1})\chi_{i,j}
=
\widetilde f^{(2)}(t_{i-1})
\sum_{j=1}^{i}\chi_{i,j}
-
\sum_{j=1}^{i-1}
\bigl(
\widetilde f^{(2)}(t_j)-\widetilde f^{(2)}(t_{j-1})
\bigr)
\sum_{k=1}^{j}\chi_{i,k}.
\end{equation}
Since $\widetilde f^{(3)}$ is bounded, we have $| \widetilde f^{(2)}(t_j)-\widetilde f^{(2)}(t_{j-1}) |
\le
\max_{0\le t\le T}|\widetilde f^{(3)}(t)|\,\Delta t$, and
\begin{align}
\Big|\widetilde f^{(2)}(t_{i})& \psi_i - \sum_{j=1}^{i}\widetilde f^{(2)}(t_{j-1}) \chi_{i,j}\Big| \\
\le  &  \left|\widetilde f^{(2)}(t_{i}) \left(\psi_i - \sum_{j=1}^{i}\chi_{i,j}\right)\right| + \left|\widetilde f^{(2)}(t_{i}) - \widetilde f^{(2)}(t_{i-1})  \right| \sum_{j=1}^{i}\chi_{i,j} \nonumber \\
& +\left |\sum_{j=1}^{i-1} (\widetilde f^{(2)}(t_{j}) -\widetilde f^{(2)}(t_{j-1}))\sum_{k=1}^{j} \chi_{i,k} \right | \nonumber\\
=& \mathcal{O}(i^{-3/2}) + \mathcal{O}(\Delta t i^{-1/2}) +\mathcal{O}(\Delta t ). \nonumber
\end{align}
Since $\Delta t=T/m\le T/i$, and $W_{i+1,0}^{(1)}-W_{i,0}^{(1)}
=
\mathcal{O}(i^{-3/2})$,
it follows that
\begin{equation}
| \mathcal S_{i+1}-\mathcal S_i | = \mathcal{O}(\Delta t^{3/2} i^{-3/2}).
\end{equation}
Dividing by $W_{i+1,i}^{(2)}=\mathcal{O}(i^{-1/2})$, we finally get
\begin{equation}
|b_i|
=
\frac{|\mathcal S_{i+1}-\mathcal S_i|}{W_{i+1,i}^{(2)}}
= \mathcal{O}(\Delta t^{3/2} i^{-1}) \le \varrho_i K.
\end{equation}
Therefore, the assumptions of Lemma~\ref{Kantorovich-Krylov} are satisfied, which completes the proof for the case of a constant kernel. The extension to a non-constant kernel $k(t,u)$ follows by applying the approach developed in \cite{WA-1971}.
\qed


\begin{thebibliography}{99}


\bibitem[Abundo (2003)]{A-2003}
Abundo, M. 2003. On the first-passage time of diffusion processes over a one-sided stochastic boundary. \textit{Stochastic Analysis and Applications}, 21(1), 1--23. 

\bibitem[Alili, Patie and Pedersen (2005)]{APP-2005}
Alili, L., Patie, P. and Pedersen, J.L. 2005. Representations of the first hitting time density of an Ornstein-Uhlenbeck process. \textit{Stochastic Models}, 21(4), 967--980.

\bibitem[Atkinson (1974)]{A-1974}
Atkinson, K.E. 1974. An existence theorem for Abel integral equations. \textit{SIAM J. Math. Anal.}, 5, 729--736.

\bibitem[Avellaneda and Zhu (2001)]{AZ-2001}
Avellaneda, M. and Zhu, J. 2001. Distance to default. \textit{Risk}, 14(12), 125--129.

\bibitem[Benedetto, Sacerdote and Zucca (2013)]{BSZ-2013}
Benedetto, E., Sacerdote, L. and Zucca, C. 2013. A first passage problem for a bivariate diffusion process: numerical solution with an application to neuroscience. \textit{Journal of Computational and Applied Mathematics}, 245, 10--22.

\bibitem[Ben-Hariz, Esstafa and Zaatra (2025)]{BEZ-2025}
Ben-Hariz, S., Esstafa, Y. and Zaatra, H. 2025. Recursive algorithm for transition density approximation and simulation of diffusion processes. \textit{Statistica Neerlandica}, 79(4), e70020.

\bibitem[Black and Cox (1976)]{BC-1976}
Black, F. and Cox, J.C. 1976. Valuing corporate securities: some effects of bond indenture provisions. \textit{Journal of Finance}, 31(2), 351--367.

\bibitem[Boehm et al. (2021)]{BCGS-2021}
Boehm, U., Cox, S., Gantner, G. and Stevenson, R. 2021. Fast solutions for the first-passage distribution of diffusion models with space-time-dependent drift functions and time-dependent boundaries. \textit{Journal of Mathematical Psychology}, 105, 102613.

\bibitem[Blanes et al. (2009)]{BCOR-2009}
Blanes, S., Casas, F., Oteo, J.A. and Ros, J. 2009. The Magnus expansion and some of its applications. \textit{Physics Reports}, 470(5--6), 151--238.

\bibitem[Bormetti et al. (2018)]{BCLP-2018}
Bormetti, G., Callegaro, G., Livieri, G. and Pallavicini, A. 2018. A backward Monte Carlo approach to exotic option pricing. \textit{European Journal of Applied Mathematics}, 29(1), 146--187.

\bibitem[Brunner (2004)]{B-2004}
Brunner, H. 2004. Collocation methods for Volterra integral and related functional equations. \textit{Cambridge University Press}.  

\bibitem[Brunner (2017)]{B-2017}
Brunner, H. 2017. Volterra Integral Equations: An Introduction to Theory and Applications. \textit{Cambridge University Press}. 

\bibitem[Buonocore, Nobile and Ricciardi (1987)]{BNR-1987}
Buonocore, A., Nobile, A.G. and Ricciardi, L.M. 1987. A new integral equation for the evaluation of first-passage-time probability densities. \textit{Advances in Applied Probability}, 19(4), 784--800.

\bibitem[Caflisch and Moskowitz (1995)]{CM-1995}
Caflisch, R.E. and Moskowitz, B. 1995. Modified Monte Carlo methods using quasi-random sequences. In: \textit{Monte Carlo and Quasi-Monte Carlo Methods in Scientific Computing}. Lecture Notes in Statistics, Springer, New York.

\bibitem[Detemple, Kitapbayev and Shabalin (2025)]{DKS-2025}
Detemple, J., Kitapbayev, Y. and Shabalin, D. 2025.
On the pricing of double barrier options under stochastic volatility models: A probabilistic approach.
\textit{AIMS Mathematics}, 10(9), 22622--22649.

\bibitem[Di Nardo et al. (2001)]{DNPR-2001}
Di Nardo, E., Nobile, A.G., Pirozzi, E. and Ricciardi, L.M. 2001. A computational approach to first-passage-time problems for Gauss-Markov processes. \textit{Advances in Applied Probability}, 33(2), 453--482.

\bibitem[Durbin (1971)]{D-1971}
Durbin, J. 1971. Boundary-crossing probabilities for the Brownian motion and Poisson processes and techniques for computing the power of the Kolmogorov-Smirnov Test. \textit{Journal of Applied Probability}, 8(3), 431--453.

\bibitem[Durbin and Williams (1992)]{DW-1992}
Durbin, J. and Williams, D. 1992. The first-passage density of the Brownian motion process to a curved boundary. \textit{Journal of Applied Probability}, 29(2), 291--304.

\bibitem[Eggermont (1981)]{E-1981}
Eggermont, P.P.B. 1981. A new analysis of the trapezoidal-discretization method for the numerical solution of Abel-type integral equations. \textit{Journal of Integral Equations}, 3(4), 317--332.

\bibitem[Fermo and Occorsio (2022)]{FO-2022}
Fermo, L. and Occorsio, D. 2022. Weakly singular linear Volterra integral equations: A Nyström method in weighted spaces of continuous functions. \textit{J. Comput. Appl. Math.}, 406, 114001.

\bibitem[Fortet (1943)]{F-1943}
Fortet, R. 1943. Les fonctions aléatoires du type de Markoff associées à certaines équations linéaires aux dérivées partielles du type parabolique. \textit{Journal de Mathématiques Pures et Appliquées}, 9th series, volume 22, 177--243.


\bibitem[Friedman (1964)]{F-1964}
Friedman, A. 1964. \textit{Partial Differential Equations of Parabolic Type}. Prentice-Hall, Englewood Cliffs, NJ.

\bibitem[Giorno et al. (1989)]{GNRS-1989}
Giorno, V., Nobile, A.G., Ricciardi, L.M. and Sato, S. 1989. On the evaluation of first-passage-time probability densities via non-singular integral equations. \textit{Advances in Applied Probability}, 21(1), 20--36.


\bibitem[Gobet (2000)]{G-2000}
Gobet, E. 2000. Weak approximation of killed diffusion using Euler schemes. \textit{Stochastic Processes and their Applications}, 87(2), 167--197.

\bibitem[Goovaerts, De Schepper and Decamps (2004)]{GDD-2004}
Goovaerts, M., De Schepper, A. and Decamps, M. 2004. Closed-form approximations for diffusion densities: a path integral approach. \textit{Journal of Computational and Applied Mathematics}, 164--165, 337--364.

\bibitem[Guti{\'{e}}rrez et al. (1997)]{GRRT-1997}
Guti{\'{e}}rrez, R., Ricciardi, L.M., Rom{\'{a}}n, P. and Torres, F. 1997. First-passage-time densities for time-non-homogeneous diffusion processes. \textit{Journal of Applied Probability}, 34(3), 623--631.

\bibitem[Herrmann and Zucca (2019)]{HZ-2019}
Herrmann, S. and Zucca, C. 2019. Exact simulation of the first-passage time of diffusions. \textit{J. Sci. Comput.}, 79(3), 1477--1504.

\bibitem[Higham (2005)]{H-2005}
Higham, N.J. 2005. The scaling and squaring method for the matrix exponential revisited. \textit{SIAM J. Matrix Anal. Appl.}, 26(4), 1179--1193.

\bibitem[Hu, Cheng and Berne (2010)]{HCB-2010}
Hu, Z., Cheng, L. and Berne, B.J. 2010. First passage time distribution in stochastic processes with moving and static absorbing boundaries with application to biological rupture experiments. \textit{Journal of Chemical Physics}, 133(3), 034105.

\bibitem[Ichiba and Kardaras (2011)]{IK-2011}
Ichiba, T. and Kardaras, C. 2011. Efficient estimation of one-dimensional diffusion first passage time densities via Monte Carlo simulation. \textit{J. Appl. Probab.}, 48(3), 699--712.


\bibitem[Itkin and Muravey (2022)]{IM-2022}
Itkin, A. and Muravey, D. 2022. Semi-analytic pricing of double barrier options with time-dependent barriers and rebates at hit. \textit{Frontiers of Mathematical Finance}, 1(1), 53--79.

\bibitem[Kantorovich and Krylov (1958)]{KK-1958}
Kantorovich, L.V. and Krylov, V.I. 1958. Approximate Methods of Higher Analysis. New York: Interscience.

\bibitem[Karatzas and Shreve (1991)]{KS-1991}
Karatzas, I. and Shreve, S.E. 1991. \textit{Brownian Motion and Stochastic Calculus}. 2nd ed., Graduate Texts in Mathematics, Vol. 113. Springer, New York, NY.

\bibitem[Kolmogorov (1931)]{K-1931}
Kolmogorov, A.N. 1931. Über die analytischen Methoden in der Wahrscheinlichkeitsrechnung. \textit{Mathematische Annalen}, 104, 415--458.

\bibitem[Krylov and Zvonkin (1981)]{KZ-1981}
Krylov, N.V. and Zvonkin, A.K. 1981. On strong solutions of stochastic differential equations. \textit{Sel. Math. Sov.}, 1(1), 19--61.

\bibitem[Kurtz (1969)]{K-1969}
Kurtz, T.G. 1969. Extensions of Trotter's operator semigroup approximation theorems. \textit{Journal of Functional Analysis}, 3(3), 354--375.

\bibitem[Lacker, Shkolnikov and Zhang (2020)]{LSZ-2020}
Lacker, D., Shkolnikov, M. and Zhang, J. 2020. Inverting the Markovian projection, with an application to local stochastic volatility models. \textit{Ann. Probab.}, 48(5), 2189--2211.

\bibitem[Leland (1994)]{L-1994}
Leland, H. 1994. Corporate debt value, bond covenants, and optimal capital structure. \textit{Journal of Finance}, 49(4), 1213--1252.

\bibitem[Liang and Borovkov (2023)]{LB-2023}
Liang, V. and Borovkov, K. 2023. On Markov chain approximations for computing boundary crossing probabilities of diffusion processes. \textit{Journal of Applied Probability}, 60(4), 1386--1415.

\bibitem[Linetsky (2004)]{L-2004}
Linetsky, V. 2004. Computing hitting time densities for CIR and OU diffusions: Applications to mean-reverting models. \textit{Journal of Computational Finance}, 7(4), 1--22.

\bibitem[Linz (1985)]{L-1985}
Linz, P. 1985. \textit{Analytical and numerical methods for Volterra equations}. Philadelphia: SIAM.


\bibitem[Lipton and Kaushansky (2020a)]{LK-2020}
Lipton, A. and Kaushansky, V. 2020a. On the first hitting time density for a reducible diffusion process. \textit{Quantitative Finance}, 20(5), 723--743.

\bibitem[Lipton and Kaushansky (2020b)]{LK2-2020}
Lipton, A. and Kaushansky, V. 2020b. On three important problems in Mathematical Finance. \textit{The Journal of Derivatives}, 28(1), 123--142.


\bibitem[Loeper and Pironneau (2009)]{LP-2009}
Loeper, G. and Pironneau, O. 2009. A mixed PDE/Monte-Carlo method for stochastic volatility models. \textit{Comptes Rendus Mathématique}, 347(9--10), 559--563.

\bibitem[Masoliver and Perelló (2009)]{MP-2009}
Masoliver, J. and Perelló, J. 2009. First-passage and risk evaluation under stochastic volatility. \textit{Physical Review E}, 80(1), 016108.

\bibitem[Mijatovi{\'{c}} (2010)]{M-2010}
Mijatovi{\'{c}}, A. 2010. Local time and the pricing of time-dependent barrier options. \textit{Financ. Stoch.}, 14(1), 13--48.

\bibitem[Pagliarani and Pascucci (2012)]{PP-2012}
Pagliarani, S. and Pascucci, A. 2012. Analytical approximation of the transition density in a local volatility model. \textit{Central European Journal of Mathematics}, 10(1), 250--270.

\bibitem[Park and Schuurmann (1976)]{PS-1976}
Park, C. and Schuurmann, F.J. 1976. Evaluations of barrier crossing probabilities of Wiener paths. \textit{Journal of Applied Probability}, 13(2), 267--275.

\bibitem[Patie and Winter (2008)]{PW-2008}
Patie, P. and Winter, C. 2008. First exit time probability for multidimensional diffusions: A PDE-based approach. \textit{J. Comput. Appl. Math.}, 222, 42--53.

\bibitem[Peskir (2002)]{P-2002}
Peskir, G. 2002. On integral equations arising in the first-passage problem for Brownian motion. \textit{Journal of Integral Equations and Applications}, 14(4), 397--423.

\bibitem[Peskir (2005)]{P-2005}
Peskir, G. 2005. A change-of-variable formula with local time on curves. \textit{J. Theor. Probab.}, 18(3), 499--535.

\bibitem[Redner (2001)]{R-2001}
Redner, S. 2001. \textit{A Guide to First-Passage Processes}. Cambridge University Press, Cambridge, UK.

\bibitem[Reghai, Boya and Vong (2012)]{RBV-2012}
Reghai, A., Boya, G. and Vong, G. 2012. Local volatility: smooth calibration and fast usage. \textit{Available at SSRN}.

\bibitem[Ricciardi and Sato (1988)]{RS-1988}
Ricciardi, L.M. and Sato, S. 1988. First-passage-time density and moments of the Ornstein-Uhlenbeck process. \textit{Journal of Applied Probability}, 25(1), 43--57.

\bibitem[Rom\'an, Serrano and Torres (2008)]{RST-2008}
Rom\'an, P., Serrano, J.J. and Torres, F. 2008. First-passage-time location function: Application to determine first-passage-time densities in diffusion processes. \textit{Computational Statistics and Data Analysis}, 52(8), 4132--4146.

\bibitem[Saad (1992)]{S-1992}
Saad, Y. 1992. Analysis of some Krylov subspace approximations to the matrix exponential operator. \textit{SIAM J. Numer. Anal.}, 29(1), 209--228.

\bibitem[Sacerdote and Tomassetti (1996)]{ST-1996}
Sacerdote, L. and Tomassetti, F. 1996. On evaluations and asymptotic approximations of first-passage-time probabilities. \textit{Advances in Applied Probability}, 28(1), 270--284.

\bibitem[Schroder (1989)]{S-1989}
Schroder, M. 1989. Computing the constant elasticity of variance option pricing formula. \textit{The Journal of Finance}, 44(1), 211--219.

\bibitem[Siegert (1951)]{S-1951}
Siegert, A.J.F. 1951. On the first passage time probability problem. \textit{Physical Review}, 81(4), 617--623.

\bibitem[Su, Tretyakov and Newton (2025)]{STN-2025}
Su, H., Tretyakov, M.V. and Newton, D.P. 2025. Deep learning of transition probability densities for stochastic asset models with applications in option pricing. \textit{Management Science}, 71(4), 2922--2952.

\bibitem[Wang, Liu and Zhang (2023)]{WNG-2023}
Wang, T., Liu, S. and Zhang, Z. 2023. Singular expansions and collocation methods for generalized Abel integral equations. \textit{Journal of Computational and Applied Mathematics}, 429, 115240.

\bibitem[Weiss and Anderssen (1971)]{WA-1971}
Weiss, R. and Anderssen, R.S. 1971. A product integration method for a class of singular first kind Volterra equations. \textit{Numer. Math.}, 18, 442--456.

\bibitem[Weiss (1972)]{W-1972}
Weiss, R. 1972. Product integration for the generalized Abel equation. \textit{Math. Comp.}, 26(117), 111--123.

\bibitem[Willard (1997)]{W-1997}
Willard, G.A. 1997. Calculating prices and sensitivities for path-independent derivative securities in multifactor models. \textit{Journal of Derivatives}, 5(1), 45--61.

\bibitem[Wong and Zakai (1965)]{WZ-1965}
Wong, E. and Zakai, M. 1965. On the convergence of ordinary integrals
to stochastic integrals. \textit{Ann. Math. Statist.}, 36(5), 1560--1564.

\bibitem[Zhang and Dudko (2016)]{ZD-2016}
Zhang, Y. and Dudko, O. 2016. First-passage processes in the genome. \textit{Annual Review of Biophysics}, 45(1), 117--134.

\bibitem[Zucca and Sacerdote (2009)]{ZS-2009}
Zucca, C. and Sacerdote, L. 2009. On the inverse first-passage-time problem for a Wiener process. \textit{Ann. Appl. Probab.}, 19(4), 1319--1346.



\end{thebibliography}
\end{document}